\documentclass[leqno,onefignum,onetabnum]{siamltex1213}
\usepackage[english]{babel}
\usepackage{amsmath}
\usepackage{amssymb}
\usepackage{esint}
\usepackage{color}

\usepackage{babel}

\usepackage{graphicx}
\usepackage{epstopdf}
\usepackage{caption}

\newtheorem{remark}[theorem]{Remark}

\title{Well-Posedness of the Limiting Equation of a Noisy Consensus Model in Opinion Dynamics}

\author{ 
Bernard Chazelle\thanks{Department of Computer Science, Princeton University, Princeton, NJ 08540. (\email{chazelle@cs.princeton.edu}). This author was supported in part by NSF grants CCF-0963825, CCF-1016250 and CCF-1420112}  \and
Quansen Jiu\thanks{School of Mathematical Sciences, Capital Normal University, Beijing 100048, P. R. China. (\email{jiuqs@mail.cnu.edu.cn}). This author is partially supported by National Natural Sciences Foundation of China (No. 11171229, 11231006) and Project of Beijing Chang Cheng Xue Zhe.}  \and 
Qianxiao Li\thanks{The Program in Applied and Computational Mathematics, Princeton University, Princeton, NJ 08540. (\email{qianxiao@math.princeton.edu}). This author is supported by the Agency for Science, Technology and Research, Singapore.}  \and 
Chu Wang\thanks{The Program in Applied and Computational Mathematics, Princeton University, Princeton, NJ 08540. (\email{chuw@math.princeton.edu}).}  
}

\begin{document}
\maketitle
\slugger{sima}{xxxx}{xx}{x}{x--x}

\begin{abstract}
This paper establishes the global well-posedness of the nonlinear Fokker-Planck equation for a noisy version of the Hegselmann-Krause model. The equation captures the mean-field behavior of a classic multiagent system for opinion dynamics.
We prove the global existence, uniqueness, nonnegativity and regularity of the weak solution. We also exhibit a global stability condition, which delineates a forbidden region for consensus formation. This is the first nonlinear stability result derived for the Hegselmann-Krause model.
\end{abstract}




\section{Introduction}

Network-based dynamical systems feature agents that
communicate via a dynamic graph
while acting on the information they receive.  These systems
have received increasing attention lately because of their versatile use in modeling 
social and biological systems~\cite{axelrod1997complexity,blondel2009krause,
	castellano2009statistical,chazelle2012dynamics,chazelle2015Algo,easley2012networks,jadbabaieLM03}.
Typically, they consist of a fixed number $N$ of agents, 
each one located at $x_k(t)$ on the real line.  The agents' positions evolve 
as interactions take place along the edges of a dynamical graph 
that evolves endogenously.
The motivation behind the model is to get a better understanding of the
dynamics of collective behavior.
Following~\cite{DBLP:journals/corr/GarnierPY15,motsch2014heterophilious}, we 
express the system as a set of $N$ coupled stochastic differential equations:

\begin{equation}\label{SDE}
\mathrm{d}x_i=-\frac{1}{N}\sum_{j=1}^Na_{ij}(x_i-x_j)\mathrm{d}t+\sigma\mathrm{d}W^{(i)}_t,
\end{equation}
where $\sigma$ is the magnitude of the noise, $W_t^{(i)}$ are independent Wiener processes,
and the ``influence" parameter $a_{ij}$ is a function of the distance between
agents $i$ and $j$; in other words,  $a_{ij}=a(|x_i-x_j|)$,
where $a$ is nonnegative (to create attractive forces) and compactly supported 
over a fixed interval (to keep the range of the forces finite).  
Intuitively, the model mediates the competing tension between two
opposing forces: the sum in~(\ref{SDE}) pulls the agents toward one another while
the diffusion term keeps them jiggling in a Brownian motion; the two terms push
the system into ordered and disordered states respectively.
In the mean field limit, $N\rightarrow +\infty$, Equation~(\ref{SDE}) induces a nonlinear Fokker-Planck equation
for the agent density profile $\rho(x,t)$~\cite{DBLP:journals/corr/GarnierPY15}:
\begin{equation}\label{PDE}
\rho_{t}(x,t)=\bigg{(}\rho(x,t)\int\rho(x-y,t)y a(|y|)\mathrm{d}y\bigg{)}_{x}+\frac{\sigma^2}{2} \rho_{xx}(x,t).
\end{equation}
The function $\rho(x,t)$ is the limit density of
$\rho^N(x,t) := \frac{1}{N}\sum \delta_{x_j(t)}(dx)$,
as $N$ goes to infinity, where $\delta_x(dx)$ denotes
the Dirac measure with point mass at $x$.

In the classic Hegselmann-Krause ({\em HK}\,) model, one of the most popular systems in 
consensus dynamics~\cite{fortunato2005consensus,hegselmann2002opinion,krause2000discrete},
each one of the $N$ agents moves, at each time step,
to the mass center of all the others within a fixed distance. The position of
an agent represents its ``opinion''. If we add noise to this process,
we obtain the discrete-time version of~(\ref{SDE}) for 
$a(y)=\mathbf{1}_{\left[0,R\right]}\left(y\right)$.
To be exact, the original {\em HK} model
does not scale $a_{ij}$ by $1/N$ but by the reciprocal of the number
agents within distance $R$ of agent~$i$. Canuto et al.~\cite{canutoFT2012} have argued 
that this difference has a minor impact on the dynamics.
By preserving the pairwise symmetry among the agents, however,
the formulation~(\ref{SDE}) simplifies the analysis.

The {\em HK} model has been the subject of extensive investigation.
A sample of the literature includes work on convergence and consensus 
properties~\cite{axelrod1997complexity,bhattacharyya2013convergence,
	fortunato2005consensus,hendrickx2006convergence,lorenz2005stabilization,
	martinez2007synchronous,moreau2005stability,touri2011discrete},
conjectures about the spatial features of the attractor set~\cite{blondel20072r}, and
various extensions such as {\em HK} systems with 
inertia~\cite{DBLP:journals/corr/ChazelleW15},
leaders~\cite{hegselmann2015opinion,ramirez2006follow,wongkaew2015control},
or random jumps~\cite{pineda2013noisy}.
This paper focuses on the analysis of the nonlinear Fokker-Planck equation for the noisy {\em HK} system,
which corresponds to setting $a(y)=\mathbf{1}_{\left[0,R\right]}\left(y\right)$
in~(\ref{PDE}). 

For concreteness, let us denote $U=\left[-\ell,\ell\right]$, $U_{T}=\left[-\ell,\ell\right]\times\left(0,T\right)$
for $T>0$, and consider the following periodic problem for the {\em HK} system:  
\begin{align}
	& \begin{cases}
		\begin{aligned}\rho_{t}-\frac{\sigma^{2}}{2}\rho_{xx} & =\left(\rho G_{\rho}\right)_{x} &  & \text{in}\ U_{T}\\
			\rho\left(\ell,\cdot\right) & =\rho\left(-\ell,\cdot\right) &  & \text{on}\ \partial U\times\left[0,T\right]\\
			\rho & =\rho_{0} &  & \text{on}\ U\times\left\{ t=0\right\} 
		\end{aligned}
	\end{cases}\label{eq:evo_equation}
\end{align}
where 
\begin{equation}
G_{\rho}\left(x,t\right):=\int_{x-R}^{x+R} \left(x-y\right)\rho\left(y,t\right)dy
\label{eq:G_definition}
\end{equation}
and the initial condition $\rho_{0}$ is assumed to be a probability density, i.e.,
$\rho_{0}\geq 0$ and $\int_{U}\rho_{0}\left(x\right)dx=1$. 
The positive constants $\ell$, $R$ and $\sigma$ are fixed with $0<R<\ell$. Note that we have to periodically extend $\rho$ outside of $U$ in order to make sense of the integral above.  The periodicity of $\rho$, together with Eq.~(\ref{eq:evo_equation}), immediately implies the normalization condition
\begin{equation}
\int_{U}\rho\left(y,t\right)dy=1\label{eq:normalization_cond}
\end{equation}
for all $0\leq t<T$.

\paragraph{Main Results.}

We establish the global well-posedness of Eq.~(\ref{eq:evo_equation}), which entails the existence, uniqueness, nonnegativity and regularity of the solution. In addition, we prove a global stability condition for the uniform solution $\rho=\frac{1}{2\ell}$, representing the state without any clustering of opinions. This gives a sufficient condition involving $R$ and $\sigma$ for which no consensus can be reached regardless of the initial condition.

The paper is organized as follows. In Section~\ref{sec:apriori}, we first derive the aforementioned global stability condition by assuming that a sufficiently smooth solution exists. More precisely, we show that as long as $\sigma^2>\frac{2\ell}{\pi} \left(2R+{R^2}/{\sqrt{3}\, \ell}\right)$, the uniform solution $\rho=\frac{1}{2\ell}$ is unconditionally stable in the sense that $\rho\left(t\right)\rightarrow\frac{1}{2\ell}$ exponentially as $t\rightarrow\infty$ for any initial data $\rho_0\in {L}_{per}^{2}\left(U\right)$. An important ingredient in the proof is a ${L}_{per}^1$ estimate on the solution (Lemma \ref{lem:L1}). Interestingly, this estimate immediately implies the nonnegativity of the solution while no arguments using maximum principles are required. The latter may not be easy to obtain for nonlinear partial integro-differential equations, such as the one we consider here. We close the section by discussing the physical significance of the stability result and how it relates to other works in the opinion dynamics literature. 

In Section \ref{sec:existenceuniqueness}, we prove the global existence and uniqueness of the weak solution to (\ref{eq:evo_equation}) when the initial data $\rho_0\in {L}_{per}^{2}\left(U\right)$. Here, we construct approximate solutions by solving a series of linear parabolic equations obtained from (\ref{eq:evo_equation}) by replacing $\rho G_{\rho}$ with $\rho_n G_{\rho_{n-1}}$. Using energy estimates, we find that the sequence of solutions $\rho_n$ forms a Cauchy sequence in ${L}^1(0,T;{L}_{per}^{1}\left(U\right))$ and we use this strong convergence result to simplify the existence proof.

Finally, in Section \ref{sec:regularity} we establish improved regularity properties of the weak solution if $\rho_0\in H_{per}^{k}\left(U\right)$ for some $k\ge 1$. This allows us to remove the a priori smoothness assumptions in the stability and positivity results of Section \ref{sec:apriori}. The main results in this paper are then summarized in Theorem \ref{thm:mainresult}.

\paragraph{Notation.}

As customary in the literature, we often treat $\rho$ (and other functions on $U_{T}$) not as a function from $U_{T}$ to $\mathbb{R}$, but
from $\left[0,T\right]$ to a relevant Banach space. In this case, we define for each $t$, 
\begin{equation}
\left[\rho\left(t\right)\right]\left(\cdot\right):=\rho\left(\cdot,t\right).
\end{equation}
For a shorthand, we denote the usual $L^p$ norms on $U$ by
\begin{equation}
\left\Vert \rho\left(t\right)\right\Vert _{p} := \left\Vert \rho\left(t\right)\right\Vert _{L^p\left(U\right)} 
\end{equation}
Other Banach space norms will be written out explicitly. 
Since we are dealing with periodic boundary conditions, a subscript ``per'' will often be attached to the 
relevant Banach space symbols to denote the subspace of periodic functions on $U$, e.g., 
\begin{eqnarray}
{L}_{per}^{p}\left(U\right) & := & \left\{ u\in{L}^{p}\left(U\right):u\left(-\ell\right)=u\left(+\ell\right)\right\} ,\nonumber \\
H_{per}^{k}\left(U\right) & := & \left\{ u\in H^{k}\left(U\right):u\left(-\ell\right)=u\left(+\ell\right)\right\} ,
\end{eqnarray}
for $1\leq p\leq\infty$ and $k\geq1$. They are equipped with the usual norms. 

Finally, we denote by $C,C_{1},C_{2},\dots$ any generic constants (possibly depending on $R$, $\ell$ or $\sigma$) used in the various energy estimates. Their actual values may change from line to line. When they depend on the time interval under consideration, we will indicate the dependence explicitly, e.g., $C\left(T\right)$.

\section{Nonnegativity and Global Stability via A Priori Estimates\label{sec:apriori}}

First, let us assume that there exists a sufficiently smooth solution
\begin{equation}
\rho\in C^{1}\left(0,\infty;C_{per}^{2}\left(U\right)\right),
\end{equation}
to equation (\ref{eq:evo_equation}). This allows us to prove a priori energy estimates, from which we can deduce the solution's nonnegativity and derive a global
stability condition of the spatially uniform solution $\rho=\frac{1}{2\ell}$. 

We begin by setting $\psi=\rho-\frac{1}{2\ell}$ so that $\int_{U}\psi\left(y,t\right)dy=0$.
Then, $\psi$ satisfies the equation
\begin{equation}
\psi_{t}-\frac{\sigma^{2}}{2}\psi_{xx} = \left(\psi G_{\psi}\right)_{x}+\frac{1}{2\ell}\left(G_{\psi}\right)_{x}.\label{eq:psi}
\end{equation}
The other two extra terms are zero since $\int_{x-R}^{x+R}\left(x-y\right)dy=0$
for all $x$. Multiplying equation~(\ref{eq:psi}) by $\psi$ and
integrating by parts over $U$, we have
\begin{equation}
\frac{1}{2}\frac{d}{dt}\left\Vert \psi\left(t\right)\right\Vert _{2}^{2}+\frac{\sigma^{2}}{2}\left\Vert \psi_{x}\left(t\right)\right\Vert _{2}^{2}\leq\left|\int_{U}\psi_{x}\psi G_{\psi}dx\right|+\frac{1}{2\ell}\left|\int_{U}\psi_x G_{\psi}dx\right|.
\end{equation}
By the Cauchy-Schwarz inequality,
\begin{eqnarray}
\frac{1}{2}\frac{d}{dt}\left\Vert \psi\left(t\right)\right\Vert _{2}^{2}+\frac{\sigma^{2}}{2}\left\Vert \psi_{x}\left(t\right)\right\Vert _{2}^{2}
& \leq &\left\Vert G_{\psi}\left(t\right)\right\Vert _{\infty}\left\Vert \psi\left(t\right)\right\Vert _{2}\left\Vert \psi_{x}\left(t\right)\right\Vert _{2}
\nonumber \\ 
& & + \frac{1}{2\ell}\left\Vert \psi_x\left(t\right)\right\Vert _{2}\left\Vert G_{\psi}\left(t\right)\right\Vert _{2}.
\label{eq:energy_est1}
\end{eqnarray}
First let us estimate $\left\Vert G_{\psi}\left(t\right)\right\Vert _{\infty}$.
For any $x$ and $t$, we have 
\begin{eqnarray}
\left|G_{\psi}\left(x,t\right)\right| & \le & \int_{U}\left|y-x\right|\mathbf{1}_{\left\{ \left|y-x\right|\leq R\right\} }\left|\psi\left(y,t\right)\right|dy\nonumber \\
& \leq & R\left\Vert \psi\left(t\right)\right\Vert _{1}\nonumber \\
& \leq & R\left(1+\left\Vert \rho\left(t\right)\right\Vert _{1}\right).\label{eq:G_infty_est}
\end{eqnarray}
In order to proceed with the bound on $\left\Vert G_{\psi}\left(t\right)\right\Vert _{\infty}$,
we need an ${L}^{1}_{per}$ estimate of $\rho\left(t\right)$.
\begin{lemma}\label{lem:L1}
Suppose $\rho\in C^{1}\left(0,\infty;C_{per}^{2}\left(U\right)\right)$
is a solution of (\ref{eq:evo_equation}) with $\rho_{0}\in{L}_{per}^{1}\left(U\right)$,
then $\left\Vert \rho\left(t\right)\right\Vert _{1}\leq\left\Vert \rho_{0}\right\Vert _{1}$ for all $t\geq0$. 
\end{lemma}
\begin{proof}
	Let $\epsilon>0$ and define 
	\begin{equation}
	\chi_{\epsilon}\left(r\right)=\begin{cases}
	\left|r\right| & \left|r\right|>\epsilon\\
	-\frac{r^{4}}{8\epsilon^{3}}+\frac{3r^{2}}{4\epsilon}+\frac{3\epsilon}{8} & \left|r\right|\leq\epsilon
	\end{cases}. 
	\label{eq:chi_def}
	\end{equation}
	This is a convex $C^{2}$-approximation of the absolute value function
	satisfying $\left|r\right|\leq\chi_{\epsilon}\left(r\right)$. Multiplying
	$\chi_{\epsilon}^{\prime}\left(\rho\right)$ to equation (\ref{eq:evo_equation})
	and integrating by parts over $U$, we have
	\begin{equation}
	\frac{d}{dt}\int_{U}\chi_{\epsilon}\left(\rho\left(x,t\right)\right)\, dx+\frac{\sigma^{2}}{2}\left\Vert \rho_{x}\left(t\right)\left[\chi_{\epsilon}^{\prime\prime}\left(\rho\left(t\right)\right)\right]^{1/2}\right\Vert _{2}^{2}=-\int_{U}\rho G_{\rho}\chi_{\epsilon}^{\prime\prime}\left(\rho\right)\rho_{x}\, dx.
	\end{equation}
	By Cauchy-Schwarz and Young's inequality, 
	\begin{eqnarray}
	& & \frac{d}{dt}\int_{U}\chi_{\epsilon}\left(\rho\left(x,t\right)\right)\, dx+\frac{\sigma^{2}}{2}\left\Vert \rho_{x}\left(t\right)\left[\chi_{\epsilon}^{\prime\prime}\left(\rho\left(t\right)\right)\right]^{1/2}\right\Vert _{2}^{2}\nonumber \\
	& \leq & \left\Vert \rho\left(t\right)G_{\rho}\left(t\right)\left[\chi_{\epsilon}^{\prime\prime}\left(\rho\left(t\right)\right)\right]^{1/2}\right\Vert _{2}\left\Vert \rho_{x}\left(t\right)\left[\chi_{\epsilon}^{\prime\prime}\left(\rho\left(t\right)\right)\right]^{1/2}\right\Vert _{2}\nonumber \\
	& \leq & \frac{1}{2\sigma^{2}}\left\Vert \rho\left(t\right)G_{\rho}\left(t\right)\left[\chi_{\epsilon}^{\prime\prime}\left(\rho\left(t\right)\right)\right]^{1/2}\right\Vert _{2}^{2}+\frac{\sigma^{2}}{2}\left\Vert \rho_{x}\left(t\right)\left[\chi_{\epsilon}^{\prime\prime}\left(\rho\left(t\right)\right)\right]^{1/2}\right\Vert _{2}^{2}. 
	\end{eqnarray}
	Therefore, 
	\begin{eqnarray}
	\frac{d}{dt}\int_{U}\chi_{\epsilon}\left(\rho\left(x,t\right)\right)\, dx & \leq & \frac{1}{2\sigma^{2}}\left\Vert \rho\left(t\right)G_{\rho}\left(t\right)\left[\chi_{\epsilon}^{\prime\prime}\left(\rho\left(t\right)\right)\right]^{1/2}\right\Vert _{2}^{2}\nonumber \\
	& \leq & \frac{1}{2\sigma^{2}}\left\Vert G_{\rho}\left(t\right)\right\Vert _{\infty}^{2}\left\Vert \rho\left(t\right)\left[\chi_{\epsilon}^{\prime\prime}\left(\rho\left(t\right)\right)\right]^{1/2}\right\Vert _{2}^{2}.\label{eq:L1_beforelim}
	\end{eqnarray}
	Replacing $\psi$ by $\rho$ in (\ref{eq:G_infty_est}) we have 
	\begin{equation}
	\left\Vert G_{\rho}\left(t\right)\right\Vert _{\infty}\leq R\left\Vert \rho\left(t\right)\right\Vert _{1}.\label{eq:Gtilda}
	\end{equation}
	The term $\left\Vert \rho\left(t\right)\left[\chi_{\epsilon}^{\prime\prime}\left(\rho\left(t\right)\right)\right]^{1/2}\right\Vert _{2}^{2}$
	can be split into two integrals: 
	\begin{eqnarray}
	\left\Vert \rho\left(t\right)\left[\chi_{\epsilon}^{\prime\prime}\left(\rho\left(t\right)\right)\right]^{1/2}\right\Vert _{2}^{2} & = & \int_{U}\rho^{2}\chi_{\epsilon}^{\prime\prime}\left(\rho\right)\, dx\nonumber \\
	& = & \int_{U}\rho^{2}\chi_{\epsilon}^{\prime\prime}\left(\rho\right)\mathbf{1}_{\left\{ \left|\rho\right|>\epsilon\right\} }\, dx\nonumber \\
	&  & +\int_{U}\rho^{2}\chi_{\epsilon}^{\prime\prime}\left(\rho\right)\mathbf{1}_{\left\{ \left|\rho\right|\leq\epsilon\right\} }\, dx.
	\label{eq:bef_second_one}
	\end{eqnarray}
	For $\left|\rho\right|>\epsilon$, $\chi_{\epsilon}^{\prime\prime}\left(\rho\right)=0$
	by construction, and hence the first integral above is zero. The second integral is estimated as:
	\begin{eqnarray}
	\int_{U}\rho^{2}\chi_{\epsilon}^{\prime\prime}\left(\rho\right)\mathbf{1}_{\left\{ \left|\rho\right|\leq\epsilon\right\} }\, dx
	& = & \int_{U}\frac{3\rho^{2}\left(\epsilon^{2}-\rho^{2}\right)}{2\epsilon^{3}}\mathbf{1}_{\left\{ \left|\rho\right|\leq\epsilon\right\} }dx\nonumber \\
	& \leq & \int_{U}\frac{3\epsilon}{2} \, dx
	=  3\ell\epsilon . 
	\label{eq:second_one}
	\end{eqnarray}
	Therefore, by~(\ref{eq:Gtilda}, \ref{eq:bef_second_one}, \ref{eq:second_one}),
	Eq.~(\ref{eq:L1_beforelim}) becomes 
	\begin{equation*}
	\frac{d}{dt}\int_{U}\chi_{\epsilon}\left(\rho\left(x,t\right)\right)dx\leq\frac{3\ell R^{2}\epsilon}{2\sigma^{2}}\left\Vert \rho\left(t\right)\right\Vert _{1}^{2}\leq\frac{3\ell R^{2}\epsilon}{2 \sigma^{2}}\left[\int_{U}\chi_{\epsilon}\left(\rho\left(x,t\right)\right) dx\right]^{2}.
	\end{equation*}
	Applying Gr\"{o}nwall's inequality, we get
	\begin{eqnarray}
	\int_{U}\chi_{\epsilon}\left(\rho\left(x,t\right)\right)dx & \leq & \left(\int_{U}\chi_{\epsilon}\left(\rho_{0}\left(x\right)\right)dx\right).\nonumber \\
	&  & \times\exp\left[\frac{3\ell R^{2}\epsilon}{2\sigma^{2}}\int_{0}^{t}\int_{U}\chi_{\epsilon}\left(\rho\left(x,s\right)\right) dx\, ds\right].
	\label{eq:L1_gronwall}
	\end{eqnarray}
	Since $\rho$ is continuous, the integral in the exponential is finite.
	Therefore, taking the limit $\epsilon\rightarrow0$ yields
	\begin{equation}
	\left\Vert \rho\left(t\right)\right\Vert _{1}\leq\left\Vert \rho_{0}\right\Vert _{1},
	\end{equation}
	for every $t\geq0$, as required. 
\qquad\end{proof}

Incidentally, Lemma \ref{lem:L1} establishes the nonnegativity of
$\rho$. This is important because $\rho$ represents the density of opinions of individuals and, as such, is necessarily nonnegative at all times. It is interesting that a $L_{per}^1$ estimate suffices to show nonnegativity and no arguments from maximum principles are required. 
\begin{corollary}\label{cor:positivity}
If $\rho\in C^{1}\left(0,\infty;C_{per}^{2}\left(U\right)\right)$ is
a solution of (\ref{eq:evo_equation}), with $\rho_{0}\geq0$
and $\int_{U}\rho_{0}\left(x\right)dx=1$, then $\left\Vert \rho\left(t\right)\right\Vert _{1}=1$
and $\rho\left(t\right)\geq0$ in $U$ for all $t\geq0$. 
\end{corollary}
\begin{proof}
	Since $\int_{U}\rho_{0}\left(x\right)dx=1$, the normalization condition
	(\ref{eq:normalization_cond}) is satisfied for $t>0$. Applying Lemma
	\ref{lem:L1}, we have
	\begin{equation}
	1=\int_{U}\rho\left(x,t\right)dx\leq\left\Vert \rho\left(t\right)\right\Vert _{1}\leq\left\Vert \rho_{0}\right\Vert _{1}=\int_{U}\rho_{0}\, dx=1.
	\end{equation}
	Hence, $\left\Vert \rho\left(t\right)\right\Vert _{1}=1$. But
	\begin{eqnarray}
	1 & = & \int_{U}\rho\left(x,t\right)dx=\int_{U}\rho\mathbf{1}_{\left\{ \rho\geq0\right\} }dx+\int_{U}\rho\mathbf{1}_{\left\{ \rho<0\right\} }\, dx,\\
	1 & = & \int_{U}\left|\rho\left(x,t\right)\right|dx=\int_{U}\rho\mathbf{1}_{\left\{ \rho\geq0\right\} }dx-\int_{U}\rho\mathbf{1}_{\left\{ \rho<0\right\} }dx.
	\end{eqnarray}
	These equations imply that $\int_{U}\rho\mathbf{1}_{\left\{ \rho<0\right\} }dx=0$,
	and hence, $\rho\left(t\right)\geq0$ a.e. in $U$. By continuity, $\rho\left(t\right)\geq0$
	in $U$ for all $t\geq0$. 
\qquad\end{proof}

With Lemma \ref{lem:L1}, it follows from~(\ref{eq:G_infty_est}) that
\begin{equation}
\left\Vert G_{\psi}\left(t\right)\right\Vert _{\infty}\leq2R\label{eq:est_1}
\end{equation}
Next, we also have
\begin{eqnarray}
\left(G_{\psi}\left(x,t\right)\right)^2 & = & \left(\int_{x-R}^{x+R}\left(x-y\right) \psi\left(y,t\right)dy \right)^2 \nonumber \\
& \leq & \int_{x-R}^{x+R}\left(x-y\right)^2 dy \int_{x-R}^{x+R} \left(\psi\left(y,t\right)\right)^2dy \nonumber \\
& = & \frac{2}{3} R^3 \int_{x-R}^{x+R} \left(\psi\left(y,t\right)\right)^2dy.
\end{eqnarray}
Consequently, 
\begin{eqnarray}
\left\Vert G_{\psi}\left(t\right)\right\Vert^2_{2} & \leq & \frac{2}{3} R^3 \int_{U} \int_{x-R}^{x+R} \left(\psi\left(y,t\right)\right)^2 dy dx \nonumber \\
& = & \frac{2}{3} R^3 \int_{U} \int_{-R}^{+R} \left(\psi\left(x+y,t\right)\right)^2 dy dx \nonumber \\
& = & \frac{4R^4}{3} \left\Vert \psi\left(t\right)\right\Vert^2_{2}.\label{eq:est_2}
\end{eqnarray}
With estimates (\ref{eq:est_1}) and (\ref{eq:est_2}), (\ref{eq:energy_est1})
becomes
\begin{eqnarray}
\frac{1}{2}\frac{d}{dt}\left\Vert \psi\left(t\right)\right\Vert _{2}^{2}+\frac{\sigma^{2}}{2}\left\Vert \psi_{x}\left(t\right)\right\Vert _{2}^{2} & \leq & 
\left(2R+\frac{R^2}{\sqrt{3}\, \ell}\right)\left\Vert \psi\left(t\right)\right\Vert _{2}\left\Vert \psi_{x}\left(t\right)\right\Vert_{2}\nonumber \\
& \leq & \frac{\sigma^{2}}{4}\left\Vert \psi_{x}\left(t\right)\right\Vert _{2}^{2}\nonumber \\
&  & +\frac{1}{\sigma^2} \left(2R+\frac{R^2}{\sqrt{3}\, \ell}\right)^2 \left\Vert \psi\left(t\right)\right\Vert _{2}^{2}.
\end{eqnarray}
Hence, we have
\begin{equation}
\frac{1}{2}\frac{d}{dt}\left\Vert \psi\left(t\right)\right\Vert _{2}^{2}\leq
\frac{1}{\sigma^2} \left(2R+\frac{R^2}{\sqrt{3}\, \ell}\right)^2 
\left\Vert \psi\left(t\right)\right\Vert _{2}^{2}-\frac{\sigma^{2}}{4}\left\Vert \psi_{x}\left(t\right)\right\Vert _{2}^{2}.\label{eq:energy_est2}
\end{equation}
By construction, $\int_{U}\psi\left(x,t\right)dx=0$ for all $t$.
Thus, $\psi\left(t\right)$ satisfies the Poincar\'{e} inequality
\begin{equation}
\left\Vert \psi\left(t\right)\right\Vert _{2}\leq C\left\Vert \psi_{x}\left(t\right)\right\Vert _{2}.
\end{equation}
For the interval $U=\left[-\ell,\ell\right]$, the optimal Poincar\'{e}
constant is $C=\ell/\pi$. Therefore, (\ref{eq:energy_est2}) becomes
\begin{equation}
\frac{d}{dt}\left\Vert \psi\left(t\right)\right\Vert _{2}^{2}
\leq\left(\frac{2}{\sigma^2} \left(2R+\frac{R^2}{\sqrt{3}\, \ell}\right)^2 -\frac{\pi^{2}\sigma^{2}}{2\ell^{2}}\right)\left\Vert \psi\left(t\right)\right\Vert _{2}^{2}.
\end{equation}
But  
\begin{equation}
\left\Vert \psi\left(0\right)\right\Vert _{2}^{2}=\left\Vert \rho_{0}-\frac{1}{2\ell}\right\Vert _{2}^{2}\leq 2\left\Vert \rho_{0}\right\Vert _{2}^{2}+\frac{1}{\ell} \, .
\end{equation}
Thus we obtain the integral form of (\ref{eq:energy_est2}): 
\begin{equation}
\left\Vert \psi\left(t\right)\right\Vert _{2}^{2}
\leq\left(2\left\Vert \rho_{0}\right\Vert _{2}^{2}+\frac{1}{\ell}\right)\exp
\left\{{\left(\frac{2}{\sigma^2}\left(2R+\frac{R^2}{\sqrt{3}\, \ell}\right)^2-\frac{\pi^{2}\sigma^{2}}{2\ell^{2}}\right)t}\right\}.
\end{equation}
In particular, if $\sigma^2>\frac{2\ell}{\pi} \left(2R+{R^2}/{\sqrt{3}\, \ell}\right)$,
the constant factor in the exponential is negative, therefore
$\left\Vert \psi\left(t\right)\right\Vert _{2}^{2}\rightarrow0$ as
long as $\left\Vert \rho_{0}\right\Vert _{2}<\infty$. We summarize
these results:

\begin{theorem}
	Let $\rho\in C^{1}\left(0,\infty;C_{per}^{2}\left(U\right)\right)$
	be a solution of (\ref{eq:evo_equation}) with $\rho_{0}\geq0$,
	$\int_{U}\rho_{0}\left(x\right)dx=1$, and $\rho_{0}\in{L}_{per}^{2}\left(U\right)$.
	If $\sigma^2>\frac{2\ell}{\pi} \left(2R+{R^2}/{\sqrt{3}\, \ell}\right)$, 
	then $\rho\left(t\right)\rightarrow\frac{1}{2\ell}$ in ${L}_{per}^{2}$
	exponentially as $t\rightarrow\infty$. \label{thm:apriori_stability}
\end{theorem}

\subsection*{Physical Significance of Theorem~\ref{thm:apriori_stability}}
The noisy {\it HK} model generally exhibits two types of steady-states. The first is a spatially uniform steady-state, i.e., $\rho$ is constant. This represents the case where individuals have uniformly distributed opinions, without any local or global consensus. The second involves one or more clusters of individuals having similar opinions, in which case $\rho$ is a multi-modal profile. Which of these two steady-states appear in the long-time limit depends on the interaction radius $R$ and noise $\sigma$, as well as the initial profile $\rho_0$. 

In this aspect, Theorem~\ref{thm:apriori_stability} gives a sufficient condition for the spatially uniform solution to be the globally attractive state, irrespective of the initial profile $\rho_0$. In other words, as long as 
$\sigma^2>\frac{2\ell}{\pi} \left(2R+{R^2}/{\sqrt{3}\, \ell}\right)$,
any initial profile converges to the spatially uniform state. In particular, clustered profiles do not even have local stability. 
This immediately indicates a forbidden zone for consensus: when the volatility of one's opinion is too large compared to the interaction radius, there can be no clustering of opinions regardless of the initial opinion distribution. 
It should be noted that this is the first result regarding the nonlinear stability of the {\em HK} system. 
On the other hand, it is straightforward to perform linear stability analysis of Eq. (\ref{eq:evo_equation})
at the uniform solution $\rho=\frac{1}{2\ell}$ to derive a linear stability condition for the uniform solution~\cite{wang2015inprep}. The combination of these two results indicate a region where it is possible to have both clustered and uniform states as locally stable solutions (see Figure~\ref{fig:pd}). 
\begin{figure}
	\centering
	\includegraphics[width=\textwidth]{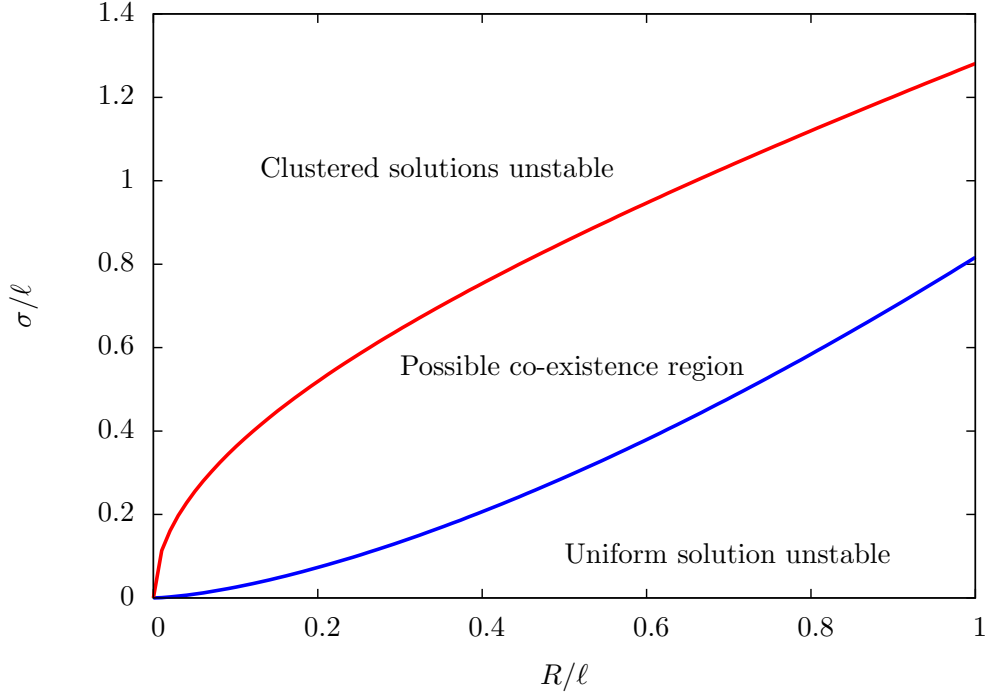}
	\caption{The phase diagram for the noisy {\it HK} model. The top (red) curve is the curve $\sigma^2=\frac{2\ell}{\pi} \left(2R+{R^2}/{\sqrt{3}\, \ell}\right)$, above which the spatially uniform solution ($\rho = 1/2\ell$) is unconditionally stable, i.e. no clustering of opinions is possible. The bottom (blue) curve is obtained from linear stability analysis around the spatially uniform solution, and has the form $(\sigma/\ell)^2 = C (R/\ell)^3$. Below this curve, the uniform solution loses linear stability and only clustered solutions are permitted. Between these two curves is the region for which both clustered and uniform solutions can be stable with respect to small perturbations. }
	\label{fig:pd}
\end{figure}

\section{Existence and Uniqueness\label{sec:existenceuniqueness}}

Our discussion so far has assumed the existence of a solution to (\ref{eq:evo_equation}). In this section, 
we prove the existence and uniqueness of the weak solution by defining
a sequence of linear parabolic equations, whose solutions converge
strongly to a function $\rho$ that solves a weak formulation of Eq.~(\ref{eq:evo_equation}). 
To begin with, let $T>0$ and consider a sequence
of linear parabolic equations 
\begin{align}
& \begin{cases}
\begin{aligned}\rho_{n_{t}}-\frac{\sigma^{2}}{2}\rho_{n_{xx}} & = \, \left(\rho_{n}G_{\rho_{n-1}}\right)_{x} &  & \text{in}\ U_{T}\\
\rho_{n}\left(\ell,\cdot\right) & =\, \rho_{n}\left(-\ell,\cdot\right) &  & \text{on}\ \partial U\times\left[0,T\right]\\
\rho_{n} & =\, \rho_{0} &  & \text{on}\ U\times\left\{ t=0\right\} 
\end{aligned}
\end{cases}\label{eq:sequence_evolution}
\end{align}
for $n\geq1$, with $\rho_{0}\left(x,t\right):=\rho_{0}\left(x\right)$
for all $t>0$. For convenience, we assume that the initial condition
satisfies $\rho_{0}\in C_{per}^{\infty}\left(U\right)$, $\rho_{0}\geq0$
and $\int_{U}\rho_0\left(x\right)dx=1$. The smoothness condition will
be relaxed later (see Theorem~\ref{thm:Existnece-1} at the end of this section). 

Consider the case $n=1$. Since $\rho_{0}\in C_{per}^{\infty}\left(U\right)$
and both $G_{\rho_{0}}$ and $\left(G_{\rho_{0}}\right)_{x}$ are bounded,
by standard results on linear parabolic evolution equations,
there exists a unique $\rho_{1}\in C^{\infty}\left(0,T;C_{per}^{\infty}\left(U\right)\right)$
satisfying (\ref{eq:sequence_evolution}) for $n=1$. Iterating this
for $n>1$ implies that there exists a sequence of smooth functions
$\left\{ \rho_{n}:n\geq1\right\} $ satisfying (\ref{eq:sequence_evolution}).
Next, we establish some uniform energy estimates on $\rho_{n}$.
\begin{proposition}
	Let $T>0$ and suppose $\left\{ \rho_{n}:n\geq1\right\} $ satisfy
	(\ref{eq:sequence_evolution}) with $\rho_{0}\in C_{per}^{\infty}\left(U\right)$.
	Then, $\left\Vert \rho_{n}\left(t\right)\right\Vert _{1}\leq\left\Vert \rho_{0}\right\Vert _{1}$
	for all $0\leq t\leq T$ and $n\geq1$. \label{prop:unif_L1_est}
	\end{proposition}
\begin{proof}
	Since we know that $\rho_{n}\left(t\right)\in C^{\infty}\left(0,T;C_{per}^{\infty}\left(U\right)\right)$
	for all $n\geq1$ and all $0\leq t\leq T$, we can proceed exactly
	as in the proof of Lemma \ref{lem:L1}. In this case, instead of (\ref{eq:L1_gronwall})
	we have 
	\begin{eqnarray}
	\int_{U}\chi_{\epsilon}\left(\rho_{n}\left(x,t\right)\right)dx & \leq & \left(\int_{U}\chi_{\epsilon}\left(\rho_{0}\left(x\right)\right)dx\right)\nonumber \\
	&  & \times\exp\left[\frac{3\ell R^{2}\epsilon}{2\sigma^{2}}\int_{0}^{t}\int_{U}\chi_{\epsilon}\left(\rho_{n-1}\left(x,s\right)\right)dxds\right].
	\end{eqnarray}
	Since $\rho_{n-1}$ is smooth, the integral in the exponential is
	finite, hence we take the limit $\epsilon\rightarrow0$ to obtain
	\begin{equation}
	\left\Vert \rho_{n}\left(t\right)\right\Vert _{1}\leq\left\Vert \rho_{0}\right\Vert _{1},
	\end{equation}
	for all $0\leq t\leq T$. \qquad\end{proof}

\begin{corollary}
	Let $T>0$ and suppose $\left\{ \rho_{n}:n\geq1\right\} $ satisfy
	(\ref{eq:sequence_evolution}) with $\rho_{0}\in C_{per}^{\infty}\left(U\right)$,
	$\rho_{0}\geq0$ and $\int_{U}\rho_{0}\left(x\right)dx=1$. Then,
	$\left\Vert \rho_{n}\left(t\right)\right\Vert _{1}=1$ and $\rho_{n}\left(t\right)\geq0$
	in $U$ for all $0\leq t\leq T$ and for all $n\geq1$. \label{cor:positivity_n}
\end{corollary}
\begin{proof}
	Since the functions $\rho_{n}$ are all periodic, we have $\int_{U}\rho_{n}\left(x,t\right)dx=1$;
	hence the proof is identical to Corollary \ref{cor:positivity}. \qquad\end{proof}

\begin{proposition}
	Let $T>0$ and suppose $\left\{ \rho_{n}:n\geq1\right\} $ satisfy
	(\ref{eq:sequence_evolution}) with $\rho_{0}\in C_{per}^{\infty}\left(U\right)$,
	$\rho_{0}\geq0$ and $\int_{U}\rho_{0}\left(x\right)dx=1$. Then,
	there exists a constant $C\left(T\right)>0$ such that 
	\begin{equation}
	{\left\Vert \rho_{n} \right\Vert}_{{L}^\infty\left(0,T;{L}^2_{per}\left(U\right)\right)} 
	+ {\left\Vert \rho_{n} \right\Vert}_{{L}^2\left(0,T;H^1_{per}\left(U\right)\right)} \leq C\left(T\right)\left\Vert \rho_{0}\right\Vert _{2}.\nonumber
	\end{equation}
	\label{prop:unif_L2_est}
\end{proposition}
\begin{proof}
	We proceed as in Section \ref{sec:apriori} by multiplying (\ref{eq:sequence_evolution}) by $\rho_{n}$ and integrating by parts. This gives us 
	\begin{eqnarray}
	\frac{1}{2}\frac{d}{dt}\left\Vert \rho_{n}\left(t\right)\right\Vert _{2}^{2}+\frac{\sigma^{2}}{2}\left\Vert \rho_{n_{x}}\left(t\right)\right\Vert _{2}^{2} 
	& \leq & \int_{U}\left|\rho_{n}\rho_{n_{x}}G_{\rho_{n-1}}\right|dx\nonumber \\
	& \leq & \left\Vert \rho_{n_{x}}\left(t\right)\right\Vert _{2} \left\Vert \rho_{n}\left(t\right)G_{\rho_{n-1}}\left(t\right)\right\Vert _{2} \nonumber \\
	& \leq & \frac{\sigma^{2}}{4}\left\Vert \rho_{n_{x}}\left(t\right)\right\Vert _{2}^{2}\nonumber \\
	&  & +\frac{1}{\sigma^2}\left\Vert \rho_{n}\left(t\right)\right\Vert _{2}^{2}\left\Vert G_{\rho_{n-1}}\left(t\right)\right\Vert _{\infty}^{2}.\label{eq:unif_l2_main1}
	\end{eqnarray}
	Using Proposition \ref{prop:unif_L1_est} and Corollary \ref{cor:positivity_n}
	we have 
	\begin{equation}\label{GrhoUB}
	\left\Vert G_{\rho_{n-1}}\left(t\right)\right\Vert _{\infty}\leq R\left\Vert \rho_{n-1}\left(t\right)\right\Vert _{1}=R,
	\end{equation}
	and hence (\ref{eq:unif_l2_main1}) becomes 
	\begin{equation}
	\frac{1}{2}\frac{d}{dt}\left\Vert \rho_{n}\left(t\right)\right\Vert _{2}^{2} + \frac{\sigma^2}{4} \left\Vert \rho_{n_{x}}\left(t\right)\right\Vert _{2}^{2} \leq \frac{R}{\sigma^2}\left\Vert \rho_{n}\left(t\right)\right\Vert _{2}^{2}, \label{eq:unif_l2_mainbelow}
	\end{equation}
	which implies, by integration, that
	\begin{equation}\label{rhonCub}
	\left\Vert \rho_{n}\left(t\right)\right\Vert _{2}^{2}\leq C\left(T\right)\left\Vert \rho_{0}\right\Vert _{2}^{2},
	\end{equation}
	for all $0\leq t\leq T$, and
	\begin{equation}
	\sup_{0\leq t\leq T}\left\Vert \rho_{n}\left(t\right)\right\Vert _{2}^{2} + {\int_{0}^{T} \left\Vert \rho_{n_x}\left(t\right)\right\Vert _{2}^{2} dt}
	\leq C\left(T\right)\left\Vert \rho_{0}\right\Vert _{2}^{2}.
	\end{equation}
	
\qquad\end{proof}

\begin{proposition}
	Let $T>0$ and suppose $\left\{ \rho_{n}:n\geq1\right\} $ satisfy
	(\ref{eq:sequence_evolution}) with $\rho_{0}\in C_{per}^{\infty}\left(U\right)$,
	$\rho_{0}\geq0$ and $\int_{U}\rho_{0}\left(x\right)dx=1$. Then,
	there exists a constant $C\left(T\right)>0$ such that 
	\begin{equation}
	{\left\Vert \rho_{n} \right\Vert}_{{L}^\infty\left(0,T;H^1_{per}\left(U\right)\right)} 
	+ {\left\Vert \rho_{n} \right\Vert}_{{L}^2\left(0,T;H^2_{per}\left(U\right)\right)} 
	\leq C\left(T\right) \left(\left\Vert \rho_{0}\right\Vert^2 _{H^1_{per}\left(U\right)} + \left\Vert \rho_{0}\right\Vert _{2}^4 \right)^{1/2}.
	\end{equation}
	\label{prop:unif_L2_est_higher}
\end{proposition}
\begin{proof}
	Multiplying equation (\ref{eq:sequence_evolution}) by $-\rho_{n_{xx}}$ and integrating by parts over $U$, 
	it follows from Cauchy-Schwarz, Young's inequality, and~(\ref{GrhoUB}) that
	\begin{eqnarray}\label{eq:regularity_exp_1}
	& & \frac{1}{2}\frac{d}{dt}\left\Vert \rho_{n_{x}}\left(t\right)\right\Vert _{2}^{2}
	+\frac{\sigma^{2}}{2}\left\Vert \rho_{n_{xx}}\left(t\right)\right\Vert _{2}^{2} \nonumber \\
	& \leq &  \int_{U}\left|\left(\rho_{n}G_{\rho_{n-1}}\right)_{x}\rho_{n_{xx}}\right|dx \nonumber \\
	& \leq &  \frac{\sigma^{2}}{4}\left\Vert \rho_{n_{xx}}\left(t\right)\right\Vert _{2}^{2}
	+C_{1}\left\Vert G_{\rho_{n-1}}\rho_{n_{x}}\left(t\right)\right\Vert _{2}^{2}
	+C_{2}\left\Vert \left(G_{\rho_{n-1}}\right)_{x}\left(t\right)\rho_{n}\left(t\right)\right\Vert _{2}^{2} \nonumber \\
	& \leq &  \frac{\sigma^{2}}{4}\left\Vert \rho_{n_{xx}}\left(t\right)\right\Vert _{2}^{2}
	+C_{1}\left\Vert \rho_{n_{x}}\left(t\right)\right\Vert _{2}^{2}
	+C_{2}\left\Vert \left(G_{\rho_{n-1}}\right)_{x}\left(t\right)\rho_{n}\left(t\right)\right\Vert _{2}^{2}.
	\end{eqnarray}
	Now, 
	\begin{eqnarray}
	\left(G_{\rho_{n-1}}\right)_{x} & = & \frac{\partial}{\partial x}\left[\int_{x-R}^{x+R}\left(x-y\right)\rho_{n-1}\left(y,t\right)dy\right] \nonumber \\
	& = & -R\left(\rho_{n-1}\left(x+R,t\right)+\rho_{n-1}\left(x-R,t\right)\right) \nonumber \\
	& & +\int_{x-R}^{x+R}\rho_{n-1}\left(y,t\right)dy. \label{eq:G_diff_first} 
	\end{eqnarray}
	By~(\ref{rhonCub}) and Morrey's inequality,
	\begin{eqnarray}
	\left\Vert \left(G_{\rho_{n-1}}\right)_{x}\left(t\right)\rho_{n}\left(t\right)\right\Vert _{2}^{2}
	& \leq & \left\Vert \rho_{n}\left(t\right)\right\Vert _{\infty}^{2}\left\Vert \left(G_{\rho_{n-1}}\right)_{x}\left(t\right)\right\Vert _{2}^{2}\nonumber \\
	& \leq & C\left\Vert \rho_{n}\left(t\right)\right\Vert _{H_{per}^{1}\left(U\right)}^{2}\left\Vert \rho_{n-1}\left(t\right)\right\Vert _{2}^{2} \nonumber \\
	& \leq & C\left(T\right)\left\Vert \rho_{0}\right\Vert _{2}^{2}\left\Vert \rho_{n}\left(t\right)\right\Vert _{H_{per}^{1}\left(U\right)}^{2}.
	\end{eqnarray}
	It follows that~(\ref{eq:regularity_exp_1}) becomes
	\begin{eqnarray}
	\frac{1}{2}\frac{d}{dt}\left\Vert \rho_{n_{x}}\left(t\right)\right\Vert _{2}^{2}+\frac{\sigma^{2}}{4}\left\Vert \rho_{n_{xx}}\left(t\right)\right\Vert _{2}^{2} & \leq & C_{1}\left\Vert \rho_{n_{x}}\left(t\right)\right\Vert _{2}^{2}\nonumber \\
	&  & +C_{2}\left(T\right)\left\Vert \rho_{0}\right\Vert _{2}^{2}\left\Vert \rho_{n}\left(t\right)\right\Vert _{H_{per}^{1}\left(U\right)}^{2}.
	\end{eqnarray}
	Integrating over $t$, we have
	\begin{eqnarray}
	& &\frac{1}{2}\, \underset{0\leq t \leq T}{\mbox{sup}}
	\left\Vert \rho_{n_{x}}(t)\right\Vert _{2}^{2}+\frac{\sigma^2}{4}\left\Vert \rho_{n_{xx}}\right\Vert _{{L}^{2}\left(0,T;{L}_{per}^{2}\left(U\right)\right)}^{2}\, \nonumber \\
	&\leq&  \frac{1}{2}\left\Vert \rho_{0_{x}}\right\Vert _{2}^{2} 
	+C\left(T\right)\Bigl(\left\Vert \rho_{n}\right\Vert _{{L}^{2}\left(0,T;H_{per}^{1}\left(U\right)\right)}^{2}+\left\Vert \rho_{0}\right\Vert _{2}^{2}\left\Vert \rho_{n}\right\Vert _{{L}^{2}\left(0,T;H_{per}^{1}\left(U\right)\right)}^{2}\Bigr). \nonumber \\
	\end{eqnarray}
	Applying the estimates in Proposition \ref{prop:unif_L2_est}, we find that
	\begin{equation}
	\underset{0\leq t \leq T}{\mbox{sup}}
	\left\Vert \rho_{n_{x}}\left(t\right)\right\Vert _{2}^{2}+\left\Vert \rho_{n_{xx}}\right\Vert _{{L}^{2}\left(0,T;{L}_{per}^{2}\left(U\right)\right)}^{2} \ \\
	\leq  C\left(T\right)\left(\left\Vert \rho_{0}\right\Vert _{H_{per}^{1}\left(U\right)}^{2}+\left\Vert \rho_{0}\right\Vert _{2}^{4}\right).\label{eq:reg_1}
	\end{equation}
\qquad\end{proof}

With the uniform estimates above, we can now show that $\rho_{n}$ converges strongly to a limit. 
\begin{lemma}
	Let $T>0$ and suppose that $\left\{ \rho_{n}:n\geq1\right\} $ satisfies
	(\ref{eq:sequence_evolution}) with $\rho_{0}\in C_{per}^{\infty}\left(U\right)$,
	$\rho_{0}\geq0$ and $\int_{U}\rho_{0}\left(x\right)dx=1$. Then
	there exists $\rho\in{L}^{1}\left(0,T;{L}_{per}^{1}\left(U\right)\right)$
	such that $\rho_{n}\rightarrow\rho$ in ${L}^{1}\left(0,T;{L}_{per}^{1}\left(U\right)\right)$.
	\label{lem:Cauchy}\end{lemma}
\begin{proof}
	We set $\phi_{n}=\rho_{n}-\rho_{n-1}$ for $n\geq 1$. For $n\geq 2$, the evolution equation for $\phi_{n}$ reads
	\begin{equation}
	\phi_{n_{t}}-\frac{\sigma^{2}}{2}\phi_{n_{xx}}=\left(\phi_{n}G_{\rho_{n-1}}+\rho_{n-1}G_{\phi_{n-1}}\right)_{x}.
	\end{equation}
	Let $\epsilon>0$. Multiplying the equation above by $\chi^{\prime}_{\epsilon}\left(\phi_{n}\right)$ (see definition~(\ref{eq:chi_def})) and integrating by parts yields
	\begin{eqnarray}
	& & \frac{d}{dt} \int_U \chi_{\epsilon}\left(\phi_{n}\right) dx
	+ \frac{\sigma^{2}}{2}
	\left\Vert \left[\chi^{\prime\prime}_{\epsilon}\left(\phi_n\left(t\right)\right)\right]^{1/2}\phi_{n_{x}}\left(t\right)\right\Vert _{2}^{2}\nonumber \\
	& \leq &\int_{U} \left\vert\chi^{\prime\prime}_{\epsilon}\left(\phi_n\right) \phi_{n_{x}}\phi_{n} G_{\rho_{n-1}}\right\vert dx
	+\int_{U}\left\vert\chi^{\prime}_{\epsilon}\left(\phi_n\right) \left(\rho_{n-1}G_{\phi_{n-1}}\right)_x\right\vert dx\nonumber \\
	& \leq & \frac{\sigma^{2}}{2} \left\Vert \left[\chi^{\prime\prime}_{\epsilon}\left(\phi_n\left(t\right)\right)\right]^{1/2}\phi_{n_{x}}\left(t\right)\right\Vert _{2}^{2} \nonumber \\
	&  & + \frac{1}{2\sigma^2} \left\Vert G_{\rho_{n-1}}\left(t\right)\right\Vert^{2}_{\infty}
	\left\Vert \left[\chi^{\prime\prime}_{\epsilon}\left(\phi_n\left(t\right)\right)\right]^{1/2} \phi_{n}\left(t\right)\right\Vert^2_{2} \nonumber \\
	&  & + \int_{U}\left\vert\chi^{\prime}_{\epsilon}\left(\phi_n\right) \left(\rho_{n-1}G_{\phi_{n-1}}\right)_x\right\vert dx. 
	\label{eq:phi_n_main1}
	\end{eqnarray}
	By Corollary~\ref{cor:positivity_n},
	$\left\Vert G_{\rho_{n-1}}\left(t\right)\right\Vert _{\infty}\leq R\left\Vert \rho_{n-1}\left(t\right)\right\Vert _{1}=R$. Also, 
	as in~(\ref{eq:second_one}) from the proof of Lemma~\ref{lem:L1},
	we have 
	\begin{equation}
	\left\Vert \left[\chi^{\prime\prime}_{\epsilon}\left(\phi_n\left(t\right)\right)\right]^{1/2}\phi_{n}\left(t\right)\right\Vert _{2}^{2}
	\leq C \epsilon.
	\end{equation} 
	To estimate the last integral in~(\ref{eq:phi_n_main1}), observe that $\left\vert\chi^{\prime}_{\epsilon}\right\vert \leq C$, hence
	\begin{eqnarray}
	\int_{U}\left\vert\chi^{\prime}_{\epsilon}\left(\phi_n\right) \left(\rho_{n-1} G_{\phi_{n-1}}\right)_x\right\vert dx 
	& \leq & C \int_{U}\left\vert \left(\rho_{n-1}\right)_x G_{\phi_{n-1}}\right\vert dx \nonumber \\
	&  & + C \int_{U}\left\vert \rho_{n-1} \left(G_{\phi_{n-1}}\right)_x \right\vert dx \nonumber \\
	& \leq & C_1 \left\Vert G_{\phi_{n-1}} \right\Vert_{\infty} \left\Vert \rho_{n-1} \right\Vert_{H^1_{per}\left(U\right)} \nonumber \\
	& & +C_2 \left\Vert \rho_{n-1} \right\Vert_{\infty} \left\Vert \left(G_{\phi_{n-1}}\right)_x \right\Vert_{1}.
	\end{eqnarray}
	But we know that $\left\Vert G_{\phi_{n-1}} \right\Vert_{\infty} \leq R\left\Vert \phi_{n-1} \right\Vert_{1}$ and that $\left\Vert \left(G_{\phi_{n-1}}\right)_x \right\Vert_{1} \leq C\left\Vert \phi_{n-1} \right\Vert_{1}$ (see expression~(\ref{eq:G_diff_first})). 
	Moreover, Morrey's inequality implies $\left\Vert \rho_{n-1} \right\Vert_{\infty} \leq \left\Vert \rho_{n-1} \right\Vert_{H^1_{per}\left(U\right)}$. 
	Hence, it follows that, as $\epsilon$ tends to $0$,  (\ref{eq:phi_n_main1}) becomes 
	\begin{eqnarray}
	\frac{d}{dt} \left\Vert \phi_{n}\left(t\right) \right\Vert_{1} 
	& \leq & C \left\Vert \rho_{n-1} \left(t\right) \right\Vert_{H^1_{per}\left(U\right)} \left\Vert \phi_{n-1} \left(t\right) \right\Vert_{1} \nonumber \\
	& \leq & C\left(\rho_0 ;1,T\right) \left\Vert \phi_{n-1} \left(t\right) \right\Vert_{1},\label{eq:phi_n_main2}
	\end{eqnarray}
	where in the last line we used Proposition~\ref{prop:unif_L2_est_higher} and the shorthand 
	\begin{equation}
	C\left(\rho_0 ;1,T\right):=C \left(T\right)\left(\left\Vert \rho_{0}\right\Vert _{H_{per}^{1}\left(U\right)}^{2}+\left\Vert \rho_{0}\right\Vert _{2}^{4}\right)^{1/2}.
	\end{equation}
	Now, for $N\geq 2$ we define 
	\begin{equation}
	y_{N}\left(t\right):=\sum_{n=2}^{N}\left\Vert \phi_{n}\left(t\right)\right\Vert _{1}.
	\end{equation}
	By~(\ref{eq:phi_n_main2}) and Corollary~\ref{cor:positivity_n},
	\begin{eqnarray}
	\frac{d}{dt}y_{N}\left(t\right) 
	& \leq & C\left(\rho_0 ;1,T\right) \left(y_N\left(t\right) + \left\Vert \phi_{1}\left(t\right)\right\Vert _{1} - \left\Vert \phi_{N}\left(t\right)\right\Vert _{1} \right) \nonumber \\
	& \leq & C\left(\rho_0 ;1,T\right) \left(y_N\left(t\right) + 4 \right). 
	\end{eqnarray}
	Moreover, the $\rho_{n}$'s coincide at $t=0$, so $y_{N}\left(0\right)=0$.
	Thus by Gr\"{o}nwall's inequality,
	\begin{equation}
	y_{N}\left(t\right) \leq 2T C\left(\rho_0 ;1,T\right) e^{T C\left(\rho_0 ;1,T\right)},
	\end{equation}
	uniformly in $N$ and $t$. Furthermore, for each $t$, $y_{N}\left(t\right)$
	is a bounded monotone sequence in $N$, hence there exists 
	\begin{equation}
	y_{\infty}\left(t\right)=\sum_{n=2}^{\infty}\left\Vert \phi_{n}\left(t\right)\right\Vert _{1}\leq 2T C\left(\rho_0 ;1,T\right) e^{T C\left(\rho_0 ;1,T\right)},
	\end{equation}
	such that $y_{N}\left(t\right)\uparrow y_{\infty}\left(t\right)$,
	pointwise in $t$. By the monotone convergence theorem, 
	\begin{equation}
	\int_{0}^{T}y_{N}\left(t\right)dt\uparrow\int_{0}^{T}y_{\infty}\left(t\right)dt\leq 2T^2 C\left(\rho_0 ;1,T\right) e^{T C\left(\rho_0 ;1,T\right)}.\label{eq:yN_convergence}
	\end{equation}
	This result immediately implies that $\left\{ \rho_{n}\right\} $
	is a Cauchy sequence in ${L}^{1}\left(0,T;{L}_{per}^{1}\left(U\right)\right)$.
	Indeed, for $\epsilon>0$ we can pick $N\geq 2$ such that $\int_{0}^{T}y_{\infty}\left(t\right)dt-\int_{0}^{T}y_{N}\left(t\right)dt<\epsilon$. Hence, for all $M\geq 1$, 
	\begin{eqnarray}
	\left\Vert \rho_{N+M}-\rho_{N}\right\Vert _{{L}^{1}\left(0,T;{L}_{per}^{1}\left(U\right)\right)} 
	& = & \int_{0}^{T}\left\Vert \rho_{N+M}\left(t\right)-\rho_{N}\left(t\right)\right\Vert _{1} dt\nonumber \\
	& = & \int_{0}^{T}\left\Vert \sum_{n=N+1}^{N+M} \phi_{n}\left(t\right)\right\Vert _{1} dt \nonumber \\
	& \leq & \int_{0}^{T}\sum_{n=N+1}^{N+M} \left\Vert \phi_{n}\left(t\right)\right\Vert _{1} dt \nonumber \\
	& = & \int_{0}^{T}y_{N+M}\left(t\right)dt-\int_{0}^{T}y_{N}\left(t\right) dt \nonumber \\
	& \leq & \int_{0}^{T}y_{\infty}\left(t\right)dt-\int_{0}^{T}y_{N}\left(t\right) dt \nonumber \\
	& \leq & \epsilon.
	\end{eqnarray}
	Therefore, $\left\{ \rho_{n}\right\} $ is a Cauchy sequence and there
	exists $\rho\in{L}^{1}\left(0,T;{L}_{per}^{1}\left(U\right)\right)$
	such that $\rho_{n}\rightarrow\rho$ in ${L}^{1}\left(0,T;{L}_{per}^{1}\left(U\right)\right)$. 
\qquad\end{proof}

Note that we can extract from $\left\{ \rho_{n}\right\} $ a subsequence
that converges weakly in smaller spaces. 
\begin{definition}
	We denote by $H_{per}^{-1}\left(U\right)$ the dual space of $H_{per}^{1}\left(U\right)$. 
\end{definition}
Since periodic boundary conditions allows integration by parts without
extra terms, most characterizations of $H^{-1}=\left(H_{0}^{1}\right)^{*}$
carries over to $H_{per}^{-1}$. 
\begin{lemma}
	We have $\rho\in{L}^{2}\left(0,T;H_{per}^{1}\left(U\right)\right)\cap{L}^{\infty}\left(0,T;{L}_{per}^{2}\left(U\right)\right)$, with $\rho_{t}\in{L}^{2}\left(0,T;H_{per}^{-1}\left(U\right)\right)$, and the estimate
	\begin{equation}
	\left\Vert \rho\right\Vert _{{L}^{\infty}\left(0,T;{L}_{per}^{2}\left(U\right)\right)}+\left\Vert \rho\right\Vert _{{L}^{2}\left(0,T;H_{per}^{1}\left(U\right)\right)}+\left\Vert \rho_{t}\right\Vert _{{L}^{2}\left(0,T;H_{per}^{-1}\left(U\right)\right)}\leq C\left(T\right)\left\Vert \rho_{0}\right\Vert _{2}.
	\end{equation}
	Moreover, there exists a subsequence $\left\{ \rho_{n_{k}}:k\geq1\right\}$ such that 
	\[
	\rho_{n_{k}}\rightharpoonup\rho \text{ in } {L}^{2}\left(0,T;H_{per}^{1}\right),
	\]
	and 	
	\[
	\rho_{n_{k_{t}}}\rightharpoonup\rho_{t} \text{ in } {L}^{2}\left(0,T;H_{per}^{-1}\right).
	\]
	\label{lem:weakconv}\end{lemma}
\begin{proof}
	From Proposition~\ref{prop:unif_L2_est}, we have 
	\begin{equation}
	\left\Vert \rho_n\right\Vert _{{L}^{\infty}\left(0,T;{L}_{per}^{2}\left(U\right)\right)}+\left\Vert \rho_n\right\Vert _{{L}^{2}\left(0,T;H_{per}^{1}\left(U\right)\right)} \leq C\left(T\right)\left\Vert \rho_{0}\right\Vert _{2}.
	\end{equation}
	Next, observe that from the evolution equation of $\rho_{n}$, we have
	\begin{equation}
	\rho_{n_{t}}=\left(G_{\rho_{n-1}}\rho_{n}+\frac{\sigma^{2}}{2}\rho_{n_{x}}\right)_{x}.
	\end{equation}
	Hence, 
	\begin{eqnarray}
	\left\Vert \rho_{n_{t}}\right\Vert _{{L}^{2}\left(0,T;H_{per}^{-1}\left(U\right)\right)}^{2} & \leq\int_{0}^{T} & \left\Vert G_{\rho_{n-1}}\left(t\right)\rho_{n}\left(t\right)+\frac{\sigma^{2}}{2}\rho_{n_{x}}\left(t\right)\right\Vert _{2}^{2}dt\nonumber \\
	& \leq & 2\int_{0}^{T}\left\Vert G_{\rho_{n-1}}\left(t\right)\right\Vert _{\infty}^{2}\left\Vert \rho_{n}\left(t\right)\right\Vert _{2}^{2}dt\nonumber \\
	&  & +\sigma^{2}\int_{0}^{T}\left\Vert \rho_{n_{x}}\left(t\right)\right\Vert _{2}^{2}dt\nonumber \\
	& \leq & 2R^2\int_{0}^{T}\left\Vert \rho_{n}\left(t\right)\right\Vert _{2}^{2}dt\nonumber \\
	&  & +\sigma^{2}\int_{0}^{T}\left\Vert \rho_{n_{x}}\left(t\right)\right\Vert _{2}^{2}dt\nonumber \\
	& \leq & C\left(T\right)\left\Vert \rho_{0}\right\Vert _{2}^{2}.
	\end{eqnarray}
	where in the last step we used Proposition~\ref{prop:unif_L2_est}. 
	Therefore, we have the uniform estimate 
	\begin{equation}
	\left\Vert \rho_{n}\right\Vert _{{L}^{\infty}\left(0,T;{L}_{per}^{2}\left(U\right)\right)}+\left\Vert \rho_{n}\right\Vert _{{L}^{2}\left(0,T;H_{per}^{1}\left(U\right)\right)}+\left\Vert \rho_{n_{t}}\right\Vert _{{L}^{2}\left(0,T;H_{per}^{-1}\left(U\right)\right)}\leq C\left(T\right)\left\Vert \rho_{0}\right\Vert _{2}.\label{eq:unif_3_ineq}
	\end{equation}
	Hence, $\rho\in{L}^{2}\left(0,T;H_{per}^{1}\left(U\right)\right)\cap{L}^{\infty}\left(0,T;{L}_{per}^{2}\left(U\right)\right)$, with
	$\rho_{t}\in{L}^{2}\left(0,T;H_{per}^{-1}\left(U\right)\right)$
	and they satisfy the same estimate (\ref{eq:unif_3_ineq}). Furthermore, there exists $\left\{ \rho_{n_{k}}:k\geq1\right\} $ such that 
	\begin{equation}
	\begin{cases}
	\rho_{n_{k}}\rightharpoonup\rho & \mbox{in }{L}^{2}\left(0,T;H_{per}^{1}\left(U\right)\right),\\
	\rho_{n_{k_{t}}}\rightharpoonup\rho_{t} & \mbox{in }{L}^{2}\left(0,T;H_{per}^{-1}\left(U\right)\right).
	\end{cases}
	\end{equation}
\qquad\end{proof}

Following \cite{Evans}, we can deduce from Lemma \ref{lem:weakconv} the following result: 
\begin{theorem}
	Suppose $\rho\in{L}^{2}\left(0,T;H_{per}^{1}\left(U\right)\right)$
	with $\rho_{t}\in{L}^{2}\left(0,T;H_{per}^{-1}\left(U\right)\right)$,
	then $\rho\in C\left(0,T;{L}_{per}^{2}\left(U\right)\right)$
	up to a set of measure zero. Further, the mapping
	\begin{equation}
	t\mapsto\left\Vert \rho\left(t\right)\right\Vert _{2}^{2}
	\end{equation}
	is absolutely continuous, with 
	\begin{equation}
	\frac{d}{dt}\left\Vert \rho\left(t\right)\right\Vert _{2}^{2}=2\left\langle \rho_{t}\left(t\right),\rho\left(t\right)\right\rangle ,
	\end{equation}
	for a.e. $0\leq t\leq T$. Here, $\left\langle ,\right\rangle $ denotes
	the pairing between $H_{per}^{-1}$ and $H_{per}^{1}$. \label{thm:Evans}\end{theorem}
\begin{proof}
	The proof is identical to the proof in Evans \cite{Evans} Section
	5.9 Theorem 3. The only difference here is that we are considering
	$H_{per}^{1}$ and $H_{per}^{-1}$, instead of $H_{0}^{1}$ and $H^{-1}$.
	Since periodic conditions still guarantees integration by parts without
	extra terms, all proofs follow through. 
\qquad\end{proof}

Now, we are ready to prove the existence of a weak solution to equation~(\ref{eq:evo_equation}). 
\begin{definition}
	We say that $\rho\in{L}^{2}\left(0,T;H_{per}^{1}\left(U\right)\right)\cap{L}^{\infty}\left(0,T;{L}_{per}^{2}\left(U\right)\right)$
	with $\rho_{t}\in{L}^{2}\left(0,T;H_{per}^{-1}\left(U\right)\right)$
	is a \emph{weak solution} of equation (\ref{eq:evo_equation}) if
	for every $\eta\in{L}^{2}\left(0,T;H_{per}^{1}\left(U\right)\right)$,
	\begin{equation}
	\int_{0}^{T}\left\langle \rho_{t}\left(t\right),\eta\left(t\right)\right\rangle dt+\int_{0}^{T}\int_{U}\left(\frac{\sigma^{2}}{2}\rho_{x}\eta_{x}+\rho G_{\rho}\eta_{x}\right)dxdt=0,
	\end{equation}
	and $\rho\left(0\right)=\rho_{0}$. Note that since $\rho\in C\left(0,T;{L}_{per}^{2}\left(U\right)\right)$
	(Theorem \ref{thm:Evans}), the last condition makes sense as an initial condition. \end{definition}

\begin{theorem}
	\label{thm:Existnece}(Existence and uniqueness) Let $\rho_{0}\in C_{per}^{\infty}\left(U\right)$,
	$\rho_{0}\geq0$ and $\int_{U}\rho_{0}\left(x\right)dx=1$. Then,
	there exists a unique weak solution $\rho\in {{L}^{\infty}\left(0,T;{L}_{per}^{2}\left(U\right)\right)}\cap {{L}^{2}\left(0,T;H_{per}^{1}\left(U\right)\right)}$, with $\rho_t \in {{L}^{2}\left(0,T;H_{per}^{-1}\left(U\right)\right)}$, to equation (\ref{eq:evo_equation})
	with the estimate 
	\[
	\left\Vert \rho\right\Vert _{{L}^{\infty}\left(0,T;{L}_{per}^{2}\left(U\right)\right)}+\left\Vert \rho\right\Vert _{{L}^{2}\left(0,T;H_{per}^{1}\left(U\right)\right)}+\left\Vert \rho_{_{t}}\right\Vert _{{L}^{2}\left(0,T;H_{per}^{-1}\left(U\right)\right)}\leq C\left(T\right)\left\Vert \rho_{0}\right\Vert _{2}.
	\]
\end{theorem}
\begin{proof}
	For each $\eta\in{L}^{2}\left(0,T;H^{1}\left(U\right)\right)$,
	we multiply equation (\ref{eq:sequence_evolution}) (with
	$n=n_{k}$) by $\eta$ and integrate over $U_{T}$ to obtain
	\begin{equation}
	\int_{0}^{T}\left\langle \rho_{n_{k_{t}}}\left(t\right),\eta\left(t\right)\right\rangle dt+\frac{\sigma^{2}}{2}\int_{0}^{T}\int_{U}\rho_{n_{k_{x}}}\eta_{x}dxdt+\int_{0}^{T}\int_{U}\eta_{x}\rho_{n_{k}}G_{\rho_{n_{k}-1}}dxdt=0.
	\end{equation}
	There are no boundary terms due to periodic boundary conditions. Now,
	\begin{eqnarray}
	\int_{0}^{T}{\it \int_{U}\eta_{x}\rho_{n_{k}}}G_{\rho_{n_{k}-1}}dxdt & = & \int_{0}^{T}{\it \int_{U}\eta_{x}\left(\rho_{n_{k}}-\rho\right)}G_{\rho_{n_{k}-1}}dxdt\nonumber \\
	&  & +\int_{0}^{T}{\it \int_{U}\eta_{x}\rho}G_{\left(\rho_{n_{k}-1}-\rho\right)}dxdt\nonumber \\
	&  & +\int_{0}^{T}{\it \int_{U}\eta_{x}\rho}G_{\rho}dxdt. \label{eq:conv0}
	\end{eqnarray}
	We know from Lemma~\ref{lem:weakconv} that $\rho_{n_k}\rightharpoonup\rho$ in ${L}^{2}\left(0,T;H_{per}^{1}\left(U\right)\right) \subset  {L}^{2}\left(0,T;{L}_{per}^{2}\left(U\right)\right)$. Moreover, $G_{\rho_{n_{k}-1}}$ is uniformly bounded so that $\eta_x G_{\rho_{n_{k}-1}}\in {L}^{2}\left(0,T;{L}_{per}^{2}\left(U\right)\right)$. Thus,
	\begin{eqnarray}
	\int_{0}^{T}{\it \int_{U}\eta_{x}\left(\rho_{n_{k}}-\rho\right)}G_{\rho_{n_{k}-1}}dxdt & \rightarrow & 0. \label{eq:conv1}
	\end{eqnarray}
	Also,
	\begin{eqnarray}
	\int_{0}^{T}{\it \int_{U}\eta_{x}\rho}G_{\left(\rho_{n_{k}-1}-\rho\right)}dxdt 
	& \leq & \left\Vert \rho\right\Vert _{{L}^{\infty}\left(0,T;{L}_{per}^{2}\left(U\right)\right)} \left\Vert \eta_{x}\right\Vert _{{L}^{2}\left(0,T;{L}_{per}^{2}\left(U\right)\right)}\nonumber \\
	&  & \left\Vert G_{\left(\rho_{n_{k}-1}-\rho\right)}\right\Vert _{{L}^{2}\left(0,T;{L}_{per}^{2}\left(U\right)\right)}\nonumber \\
	& \leq & C\left(T\right)\left\Vert \rho_{0}\right\Vert _{2}\left\Vert \eta\right\Vert _{{L}^{2}\left(0,T;H_{per}^{1}\left(U\right)\right)}\nonumber \\
	&  & \left(\int_0^T \left\Vert \rho_{n_k - 1}\left(t\right) - \rho\left(t\right) \right\Vert^{2}_{1} dt\right)^{1/2}. \label{eq:separate}
	\end{eqnarray}
	But $\left\Vert \rho_{n_k - 1}\left(t\right) - \rho\left(t\right) \right\Vert_{1}\leq 2R$. Hence, by the strong convergence result in Lemma~\ref{lem:Cauchy}, we have
	\begin{eqnarray}
	\int_0^T \left\Vert \rho_{n_k - 1}\left(t\right) - \rho\left(t\right) \right\Vert^{2}_{1} dt 
	& \leq & 2R \left\Vert \rho_{n_k - 1}\left(t\right) - \rho\left(t\right) \right\Vert_{{L}^1\left(0,T;{L}^1\left(U\right)\right)} \nonumber \\
	& \rightarrow & 0, \label{eq:strongconv}
	\end{eqnarray}
	and thus 
	\begin{eqnarray}
	\int_{0}^{T}{\it \int_{U}\eta_{x}\rho}G_{\left(\rho_{n_{k}-1}-\rho\right)}dxdt 
	& \rightarrow & 0. \label{eq:conv2}
	\end{eqnarray}
	Combining~(\ref{eq:conv0}), (\ref{eq:conv1}) and (\ref{eq:conv2}), we have
	\begin{equation}
	\int_{0}^{T}{\it \int_{U}\eta_{x}\rho_{n_{k}}}G_{\rho_{n_{k}-1}}dxdt\rightarrow\int_{0}^{T}{\it \int_{U}\eta_{x}\rho}G_{\rho}dxdt.\label{eq:conv_1}
	\end{equation}
	By the weak convergence results established in Lemma \ref{lem:weakconv},
	we also have
	\begin{eqnarray}
	\int_{0}^{T}\left\langle \rho_{n_{k_{t}}}\left(t\right),\eta\left(t\right)\right\rangle dt & \rightarrow & \int_{0}^{T}\left\langle \rho_{t}\left(t\right),\eta\left(t\right)\right\rangle dt,\nonumber \\
	\int_{0}^{T}\int_{U}\rho_{n_{k_{x}}}\eta_{x}dxdt & \rightarrow & \int_{0}^{T}\int_{U}\rho_{x}\eta_{x}dxdt.\label{eq:conv_2}
	\end{eqnarray}
	Putting together (\ref{eq:conv_1}) and (\ref{eq:conv_2}), we obtain
	in the limit $k\rightarrow\infty$, 
	\begin{equation}
	\int_{0}^{T}\left\langle \rho_{t}\left(t\right),\eta\left(t\right)\right\rangle dt+\int_{0}^{T}\int_{U}\left(\frac{\sigma^{2}}{2}\rho_{x}\eta_{x}+\rho G_{\rho}\eta_{x}\right)dxdt=0, \label{eq:weak_form}
	\end{equation}
	for every $\eta\in{L}^{2}\left(0,T;H_{per}^{1}\left(U\right)\right)$.
	
	Finally, we have to show that $\rho\left(0\right)=\rho_0$. Pick some $\eta \in C^1\left(0,T;H^1_{per}\left(U\right)\right)$ with $\eta\left(T\right)=0$. Then, we have from~(\ref{eq:weak_form}) that 
	\begin{equation}
	-\int_{0}^{T}\left\langle \rho\left(t\right),\eta_{t}\left(t\right)\right\rangle dt+\int_{0}^{T}\int_{U}\left(\frac{\sigma^{2}}{2}\rho_{x}\eta_{x}+\rho G_{\rho}\eta_{x}\right)dxdt=\left(\rho\left(0\right),\eta\left(0\right)\right). \label{eq:comp_1}
	\end{equation}
	Similarly, we also have 
	\begin{eqnarray}
	-\int_{0}^{T}\left\langle \rho_{n_k}\left(t\right),\eta_{t}\left(t\right)\right\rangle dt+\int_{0}^{T}\int_{U}\left(\frac{\sigma^{2}}{2}\rho_{{n_k}_x}\eta_{x}+\rho_{n_k} G_{\rho_{n_k-1}}\eta_{x}\right)dxdt \nonumber \\ 
	=\left(\rho_{0},\eta\left(0\right)\right). \label{eq:comp_2}
	\end{eqnarray}
	Where we have used the fact that $\rho_{n_k}\left(0\right) = \rho_0$ for all $k$. Taking the limit $k\rightarrow\infty$ and comparing~(\ref{eq:comp_1}) and~(\ref{eq:comp_2}), we have 
	\begin{equation}
	\left(\rho\left(0\right),\eta\left(0\right)\right) = \left(\rho_{0},\eta\left(0\right)\right).
	\end{equation}
	Since $\eta$ is arbitrary, we conclude that $\rho\left(0\right)=\rho_0$. This completes the proof of the existence of a weak solution. 
	
	Now, we prove its uniqueness. Let $\rho_{1}$ and $\rho_{2}$ be weak
	solutions to (\ref{eq:evo_equation}) and set $\xi=\rho_{1}-\rho_{2}$.
	Then, for every $\eta\in{L}^{2}\left(0,T;H_{per}^{1}\left(U\right)\right)$,
	we have 
	\begin{equation}
	\int_{0}^{T}\left\langle \xi_{t}\left(t\right),\eta\left(t\right)\right\rangle dt+\frac{\sigma^{2}}{2}\int_{0}^{T}\int_{U}\xi_{x}\eta_{x}dxdt+\int_{0}^{T}\int_{U}\left(\rho_{1}G_{\rho_{1}}-\rho_{2}G_{\rho_{2}}\right)\eta_{x}dxdt=0.
	\end{equation}
	Adding and subtracting $\int_{0}^{T}\int_{U}\rho_{2}G_{\rho_{1}}\eta_{x}dxdt$,
	we obtain
	\begin{eqnarray}
	\int_{0}^{T}\left\langle \xi{}_{t}\left(t\right),\eta\left(t\right)\right\rangle dt+\frac{\sigma^{2}}{2}\int_{0}^{T}\int_{U}\xi_{x}\eta_{x}dxdt & = & -\int_{0}^{T}\int_{U}\xi G_{\rho_{1}}\eta_{x}dxdt\nonumber \\
	&  & -\int_{0}^{T}\int_{U}\rho_{2}G_{\xi}\eta_{x}dxdt.
	\end{eqnarray}
	But, 
	\begin{eqnarray}
	\left|\int_{0}^{T}\int_{U}\xi G_{\rho_{1}}\eta_{x}dxdt\right| & \leq & \left\Vert G_{\rho_{1}}\right\Vert _{{L}^{\infty}\left(0,T;{L}_{per}^{\infty}\left(U\right)\right)}\nonumber \\
	&  & \times\left\Vert \eta_{x}\right\Vert _{{L}^{2}\left(0,T;{L}_{per}^{2}\left(U\right)\right)}\left\Vert \xi\right\Vert _{{L}^{2}\left(0,T;{L}_{per}^{2}\left(U\right)\right)}\nonumber \\
	& \leq & R \left\Vert \eta_{x}\right\Vert _{{L}^{2}\left(0,T;{L}_{per}^{2}\left(U\right)\right)}\left\Vert \xi\right\Vert _{{L}^{2}\left(0,T;{L}_{per}^{2}\left(U\right)\right)}\nonumber \\
	& \leq & \frac{\sigma^{2}}{4}\left\Vert \eta_{x}\right\Vert _{{L}^{2}\left(0,T;{L}_{per}^{2}\left(U\right)\right)}^{2}\nonumber \\
	&  & +C \left\Vert \xi\right\Vert _{{L}^{2}\left(0,T;{L}^{2}\left(U\right)\right)}^{2},
	\end{eqnarray}
	and 
	\begin{eqnarray}
	\left|\int_{0}^{T}\int_{U}\rho_{2}G_{\xi}\eta_{x}dxdt\right| & \leq & \left\Vert \rho_{2}\right\Vert _{{L}^{\infty}\left(0,T;{L}_{per}^{2}\left(U\right)\right)}\nonumber \\
	&  & \times\left\Vert G_{\xi}\right\Vert _{{L}^{2}\left(0,T;{L}_{per}^{2}\left(U\right)\right)}\left\Vert \eta_{x}\right\Vert _{{L}^{2}\left(0,T;{L}_{per}^{2}\left(U\right)\right)}\nonumber \\
	& \leq & C\left(T\right)\left\Vert \rho_{0}\right\Vert _{2}\left\Vert \eta_{x}\right\Vert _{{L}^{2}\left(0,T;{L}_{per}^{2}\left(U\right)\right)}\left\Vert \xi\right\Vert _{{L}^{2}\left(0,T;{L}_{per}^{2}\left(U\right)\right)}\nonumber \\
	& \leq & \frac{\sigma^{2}}{4}\left\Vert \eta_{x}\right\Vert _{{L}^{2}\left(0,T;{L}_{per}^{2}\left(U\right)\right)}^{2}\nonumber \\
	&  & +C\left(T\right)\left\Vert \rho_{0}\right\Vert _{2}^{2}\left\Vert \xi\right\Vert _{{L}^{2}\left(0,T;{L}_{per}^{2}\left(U\right)\right)}^{2},
	\end{eqnarray}
	so that 
	\begin{eqnarray}
	&  & \int_{0}^{T}\left\langle \xi{}_{t}\left(t\right),\eta\left(t\right)\right\rangle dt+\frac{\sigma^{2}}{2}\int_{0}^{T}\int_{U}\xi_{x}\eta_{x}dxdt\nonumber \\
	& \leq & \frac{\sigma^{2}}{2}\left\Vert \eta_{x}\right\Vert _{{L}^{2}\left(0,T;{L}^{2}\left(U\right)\right)}^{2}\nonumber \\
	&  & +\left(C_{1}\left(T\right)+C_{2}\left(T\right)\left\Vert \rho_{0}\right\Vert _{2}^{2}\right)\left\Vert \xi\right\Vert _{{L}^{2}\left(0,T;{L}^{2}\left(U\right)\right)}^{2}.
	\end{eqnarray}
	Now, set $\eta=\xi$, and use Theorem \ref{thm:Evans}, we have
	\begin{eqnarray}
	\int_{0}^{T}\frac{1}{2}\frac{d}{dt}\left\Vert \xi\left(t\right)\right\Vert _{2}^{2}dt & \leq & \left(C_{1}\left(T\right)+C_{2}\left(T\right)\left\Vert \rho_{0}\right\Vert _{2}^{2}\right)\left\Vert \xi\right\Vert _{{L}^{2}\left(0,T;{L}^{2}\left(U\right)\right)}^{2}\nonumber \\
	& = & \left(C_{1}\left(T\right)+C_{2}\left(T\right)\left\Vert \rho_{0}\right\Vert _{2}^{2}\right)\int_{0}^{T}\left\Vert \xi\left(t\right)\right\Vert _{2}^{2}dt.
	\end{eqnarray}
	Since this holds for all $T$, we must have 
	\begin{equation}
	\frac{d}{dt}\left\Vert \xi\left(t\right)\right\Vert _{2}^{2}\leq\left(C_{1}\left(T\right)+C_{2}\left(T\right)\left\Vert \rho_{0}\right\Vert _{2}^{2}\right)\left\Vert \xi\left(t\right)\right\Vert _{2}^{2},
	\end{equation}
	and hence 
	\begin{equation}
	\left\Vert \xi\left(t\right)\right\Vert _{2}\leq\left(C_{1}\left(T\right)+C_{2}\left(T\right)\left\Vert \rho_{0}\right\Vert _{2}^{2}\right)\left\Vert \xi\left(0\right)\right\Vert _{2},
	\end{equation}
	for a.e. $0\leq t\leq T$. But $\left\Vert \xi\left(0\right)\right\Vert _{2}^{2}=\left\Vert \rho_{0}-\rho_{0}\right\Vert _{2}^{2}=0$
	and the $\rho$'s are continuous in time, we have
	
	\begin{equation}
	\left\Vert \rho_{1}\left(t\right)-\rho_{2}\left(t\right)\right\Vert _{2}=0.
	\end{equation}
	for all $0\leq t\leq T$. 
	
	Finally, the energy estimate is from Lemma \ref{lem:weakconv}. 
\qquad\end{proof}

\begin{remark}
	The strong convergence result (Lemma~\ref{lem:Cauchy}) is important here because  without it, we could not have concluded that expression~(\ref{eq:strongconv}) converges to $0$, because it involves a different subsequence ${\rho_{n_k -1 }}$. Strong convergence ensures that all subsequences converge in $L^1\left(0,T;L^1_{per}\left(U\right)\right)$. 
\end{remark}

Throughout this section we assumed that the initial condition is smooth,
i.e. $\rho_{0}\in C_{per}^{\infty}\left(U\right)$. We can in fact
relax this condition to $\rho_{0}\in {L}_{per}^{2}\left(U\right)$ by mollifying the initial data. 

\begin{theorem}
	\label{thm:Existnece-1}(Existence and uniqueness with relaxed regularity assumption on the initial condition)
	Let $\rho_{0}\geq0$, $\int_{U}\rho_{0}\left(x\right)dx=1$ and $\rho_{0}\in {L}_{per}^{2}\left(U\right)$.
	Then, there exists a unique weak solution $\rho\in {{L}^{\infty}\left(0,T;{L}_{per}^{2}\left(U\right)\right)}\cap {{L}^{2}\left(0,T;H_{per}^{1}\left(U\right)\right)}$, with $\rho_t \in {{L}^{2}\left(0,T;H_{per}^{-1}\left(U\right)\right)}$ to equation (\ref{eq:evo_equation}) with the estimate 
	\[
	\left\Vert \rho\right\Vert _{{L}^{\infty}\left(0,T;{L}_{per}^{2}\left(U\right)\right)}+\left\Vert \rho\right\Vert _{{L}^{2}\left(0,T;H_{per}^{1}\left(U\right)\right)}+\left\Vert \rho_{_{t}}\right\Vert _{{L}^{2}\left(0,T;H_{per}^{-1}\left(U\right)\right)}\leq C\left(T\right)\left\Vert \rho_{0}\right\Vert _{2}.
	\]
	
\end{theorem}

\begin{proof}
	Let $\epsilon>0$ and consider the modified problem
	\begin{align}
	& \begin{cases}
	\begin{aligned}\rho_{n_{t}}^{\epsilon}-\frac{\sigma^{2}}{2}\rho_{n_{xx}}^{\epsilon} & =\left(\rho_{n}^{\epsilon}G_{\rho_{n-1}^{\epsilon}}\right)_{x} &  & \text{in}\ U_{T}\\
	\rho_{n}^{\epsilon}\left(\ell,\cdot\right) & =\rho_{n}^{\epsilon}\left(-\ell,\cdot\right) &  & \text{on}\ \partial U\times\left[0,T\right]\\
	\rho_{n}^{\epsilon} & =\rho_{0}^{\epsilon} &  & \text{on}\ U\times\left\{ t=0\right\} 
	\end{aligned}
	\end{cases}\label{eq:sequence_evolution_mollify}
	\end{align}
	where 
	\begin{equation}
	\rho_{0}^{\epsilon}\left(x\right):=\int_{U}j_{\epsilon}\left(x-y\right)\rho_{0}\left(y\right)dy,
	\end{equation}
	and $j_{\epsilon}\left(x\right)=\epsilon^{-1}j\left(\epsilon^{-1}x\right)$.
	Here, $j$ is a standard positive mollifier with compact support on
	$U$ and $\int_{U}j_{\epsilon}\left(x\right)dx$=1. 
	
	With mollification, $\rho_{0}^{\epsilon}$
	is now smooth and we can apply Theorem~\ref{thm:Existnece} to conclude that there exists a unique weak solution $\rho^{\epsilon}\in {{L}^{\infty}\left(0,T;{L}_{per}^{2}\left(U\right)\right)}\cap {{L}^{2}\left(0,T;H_{per}^{1}\left(U\right)\right)}$, with $\rho^{\epsilon}_t \in {{L}^{2}\left(0,T;H_{per}^{-1}\left(U\right)\right)}$ to equation (\ref{eq:sequence_evolution_mollify}) with the estimate 
	\begin{equation}
	\left\Vert \rho^{\epsilon}\right\Vert _{{L}^{\infty}\left(0,T;{L}_{per}^{2}\left(U\right)\right)}+\left\Vert \rho^{\epsilon}\right\Vert _{{L}^{2}\left(0,T;H_{per}^{1}\left(U\right)\right)}+\left\Vert \rho^{\epsilon}_{_{t}}\right\Vert _{{L}^{2}\left(0,T;H_{per}^{-1}\left(U\right)\right)}\leq C\left(T\right)\left\Vert \rho^{\epsilon}_{0}\right\Vert _{2}.
	\end{equation}
	But for all $\epsilon$, we have $\left\Vert \rho^{\epsilon}_{0}\right\Vert _{2} \leq \left\Vert \rho_{0}\right\Vert _{2}$. Hence, there exists $\rho\in {{L}^{\infty}\left(0,T;{L}_{per}^{2}\left(U\right)\right)}\cap {{L}^{2}\left(0,T;H_{per}^{1}\left(U\right)\right)}$, with $\rho_t \in {{L}^{2}\left(0,T;H_{per}^{-1}\left(U\right)\right)}$, satisfying 
	\begin{equation}
	\left\Vert \rho\right\Vert _{{L}^{\infty}\left(0,T;{L}_{per}^{2}\left(U\right)\right)}+\left\Vert \rho\right\Vert _{{L}^{2}\left(0,T;H_{per}^{1}\left(U\right)\right)}+\left\Vert \rho_{_{t}}\right\Vert _{{L}^{2}\left(0,T;H_{per}^{-1}\left(U\right)\right)}\leq C\left(T\right)\left\Vert \rho_{0}\right\Vert _{2},
	\end{equation}
	and a sequence $\left\{\epsilon_k\right\}$, with $\epsilon_k\rightarrow 0$, such that 
	\begin{eqnarray}
	\begin{cases}
	\rho^{\epsilon_k}\rightharpoonup\rho & \mbox{in }{L}^{2}\left(0,T;H_{per}^{1}\left(U\right)\right),\\
	\rho^{\epsilon_k}_t\rightharpoonup\rho_{t} & \mbox{in }{L}^{2}\left(0,T;H_{per}^{-1}\left(U\right)\right),
	\end{cases} \label{eq:weak_conv_mollify}
	\end{eqnarray}
	as $k\rightarrow \infty$. We now show that $\rho$ is in fact a weak solution to~(\ref{eq:evo_equation}). Since each $\rho^{\epsilon_k}$ solves the weak formulation of~(\ref{eq:sequence_evolution_mollify}) (albeit with different initial data), we have
	\begin{equation}
	\int_{0}^{T}\left\langle \rho^{\epsilon_k}\left(t\right),\eta\left(t\right)\right\rangle dt+\frac{\sigma^{2}}{2}\int_{0}^{T}\int_{U}\rho^{\epsilon_k}\eta_{x}dxdt+\int_{0}^{T}\int_{U}\eta_{x}\rho^{\epsilon_k}G_{\rho^{\epsilon_k}} dxdt=0.
	\end{equation} 
	Using~(\ref{eq:weak_conv_mollify}), we can replace $\rho^{\epsilon_k}$ by $\rho$ in the first two integrals above in the limit $k\rightarrow \infty$. Moreover, as in~(\ref{eq:conv0}), we write the last integral as
	\begin{eqnarray}
	\int_{0}^{T} \int_{U}\eta_{x}\rho^{\epsilon_k}G_{\rho^{\epsilon_k}}dxdt 
	& = & \int_{0}^{T} \int_{U}\eta_{x}\left(\rho^{\epsilon_k}-\rho\right) G_{\rho^{\epsilon_k}} dxdt\nonumber \\
	&  & +\int_{0}^{T} \int_{U}\eta_{x}\rho G_{\left(\rho^{\epsilon_k}-\rho\right)}dxdt\nonumber \\
	&  & +\int_{0}^{T} \int_{U}\eta_{x}\rho G_{\rho}dxdt. \label{eq:111}
	\end{eqnarray}
	Since $\left\Vert G_{\rho^{\epsilon_k}}\left(t\right) \right\Vert_{\infty} \leq R \left\Vert \rho^{\epsilon_k}\left(t\right) \right\Vert_{1} \leq R \left\Vert \rho^{\epsilon_k}_0 \right\Vert_{1} = R \left\Vert \rho_0 \right\Vert_{1} = R$, we have $\eta_{x} G_{\rho^{\epsilon_k}} \in {L}^2\left(0,T;{L}^2_{per}\left(U\right)\right)$ and hence
	\begin{equation}
	\int_{0}^{T} \int_{U}\eta_{x}\left(\rho^{\epsilon_k}-\rho\right) G_{\rho^{\epsilon_k}} dxdt \rightarrow 0.
	\end{equation}
	Next, we can write 
	\begin{eqnarray}
	\int_{0}^{T} \int_{U}\eta_{x}\rho G_{\left(\rho^{\epsilon_k}-\rho\right)}dxdt
	& = & \int_{0}^{T} \int_{U}\eta_{x}\rho \left[ \int_{x-R}^{x+R} \left(\rho^{\epsilon_k}-\rho\right) (x-y) dy \right] dxdt \nonumber \\
	& = & \int_{0}^{T} \int_{U} h \left(\rho^{\epsilon_k}-\rho\right) dy dt, \label{eq:h}
	\end{eqnarray}
	where we have defined
	\begin{eqnarray}
	h\left(y,t\right)
	& := &  \int_{y-R}^{y+R} \eta_{x}\left(x,t\right) \rho\left(x,t\right) (x-y) dx. \nonumber \\
	\end{eqnarray}
	Clearly, $\left\Vert h\left(t\right)\right\Vert_{\infty} \leq R \left\Vert \eta_x\left(t\right)\right\Vert_{2} \left\Vert \rho\left(t\right)\right\Vert_{2}$ so that in particular,  $h \in{{L}^2\left(0,T,{L}_{per}^2\left(U\right)\right)}$ and from~(\ref{eq:h}) we obtain
	\begin{equation}
	\int_{0}^{T} \int_{U}\eta_{x}\rho G_{\left(\rho^{\epsilon_k}-\rho\right)}dxdt \rightarrow 0. 
	\end{equation}
	Thus, we have shown that $\rho$ satisfies 
	\begin{equation}
	\int_{0}^{T}\left\langle \rho\left(t\right),\eta\left(t\right)\right\rangle dt+\frac{\sigma^{2}}{2}\int_{0}^{T}\int_{U}\rho\eta_{x}dxdt+\int_{0}^{T}\int_{U}\eta_{x}\rho G_{\rho} dxdt=0.
	\end{equation} 
	To show that $\rho\left(0\right)=\rho_0$, we again take $\eta \in C^1\left(0,T;H^1_{per}\left(U\right)\right)$ with $\eta\left(T\right)=0$. Since $\rho_0^{\epsilon_k}\rightarrow \rho_0$ uniformly, we have (c.f. expressions~(\ref{eq:comp_1}) and~(\ref{eq:comp_2})) 
	\begin{equation}
	\left(\rho\left(0\right),\eta\left(0\right)\right) = \left(\rho_{0},\eta\left(0\right)\right).
	\end{equation}
	Since $\eta$ is arbitrary, we have $\rho\left(0\right)=\rho_0$. The uniqueness follows from exactly the same argument in the proof of Theorem~\ref{thm:Existnece} and we omit writing it again here. 
\qquad\end{proof}

\section{Higher Regularity\label{sec:regularity}}

In this section, we prove improved regularity of the weak solution to (\ref{eq:evo_equation}). This allows us to put the results in Section \ref{sec:apriori} on a rigorous footing. As in the previous section, we always mollify $\rho_{0}$ by $j_{\epsilon}$ so that the resulting evolution equations (\ref{eq:sequence_evolution}) admit smooth solutions. This allows us to differentiate the equation as many times as required, and we take the limit $\epsilon\rightarrow0$ at the end. For simplicity of notation, we drop the $\epsilon$ superscripts on $\rho_n$ and implicitly assume that we perform the limit at the end. 

First, we prove a useful estimate. 
\begin{proposition}
	\label{prop:induction_estimate}Let $u,v\in C^{\infty}\left(U\right)$.
	Then for $k\geq2$ we have the estimate 
	\begin{eqnarray}
	\left\Vert uG_{v}\right\Vert _{H_{per}^{k}\left(U\right)} & \leq & C\left\Vert u\right\Vert _{H_{per}^{k}\left(U\right)}\left\Vert v\right\Vert _{H_{per}^{k-1}\left(U\right)}.
	\end{eqnarray}
\end{proposition}
\begin{proof}
	We have 
	\[
	\left\Vert uG_{v}\right\Vert _{H_{per}^{k}\left(U\right)}^{2}\leq C\left(\left\Vert uG_{v}\right\Vert _{2}^{2}+\left\Vert \left(uG_{v}\right)^{\left(k\right)}\right\Vert _{2}^{2}\right),
	\]
	where $\left(\cdot\right)^{(k)}$ denotes the $k^{\text th}$ derivative with respect to $x$. Applying the Leibniz rule, we have
	\begin{equation}
	\left(uG_{v}\right)^{\left(k\right)}=\sum_{i=0}^{k}\left(\begin{array}{c}
	k\\
	i
	\end{array}\right)u^{\left(k-i\right)}\left(G_{v}\right)^{\left(i\right)},
	\end{equation}
	But, 
	\begin{equation}
	\left(G_{v}\right)^{\left(i\right)}\left(x\right)=\begin{cases}
	G_{v}\left(x\right) & i=0\\
	-R\left[v\left(x+R\right)+v\left(x-R\right)\right]+\int_{x-R}^{x+R}v\left(y\right)dy & i=1\\
	-R\left[v^{\left(i-1\right)}\left(x+R\right)+v^{\left(i-1\right)}\left(x-R\right)\right]\\
	+v^{\left(i-2\right)}\left(x+R\right)-v^{\left(i-2\right)}\left(x-R\right) & i\geq2
	\end{cases}\label{eq:diff_G_high_order}
	\end{equation}
	Hence we have the bound
	\begin{eqnarray}
	\left\Vert \left(uG_{v}\right)^{\left(k\right)}\right\Vert _{2}^{2} & \leq & C_{0}\left\Vert u^{\left(k\right)}\right\Vert _{2}^{2}\left\Vert G_{v}\right\Vert _{\infty}^{2}\nonumber \\
	&  & +\sum_{i=1}^{k}C_{i}\left\Vert u^{\left(k-i\right)}\right\Vert _{\infty}^{2}\left\Vert \left(G_{v}\right)^{\left(i\right)}\right\Vert _{2}^{2}\nonumber \\
	& \leq & C_{0}\left\Vert u\right\Vert _{H_{per}^{k}\left(U\right)}^{2}\left\Vert v\right\Vert _{2}^{2}\nonumber \\
	&  & +\sum_{i=1}^{k}C_{i}\left\Vert u^{\left(k-i\right)}\right\Vert _{H_{per}^{1}\left(U\right)}^{2}\left\Vert \left(G_{v}\right)^{\left(i\right)}\right\Vert _{2}^{2}.
	\end{eqnarray}
	For $i\geq2$, 
	\begin{eqnarray}
	\left\Vert \left(G_{v}\right)^{\left(i\right)}\right\Vert _{2}^{2} & \leq & C\left(\left\Vert v^{\left(i-1\right)}\right\Vert _{2}^{2}+\left\Vert v^{\left(i-2\right)}\right\Vert _{2}^{2}\right)\nonumber \\
	& \leq & C\left\Vert v\right\Vert _{H_{per}^{i-1}\left(U\right)}^{2},
	\end{eqnarray}
	and for $i=1$,
	\begin{eqnarray}
	\left\Vert \left(G_{v}\right)^{\left(i\right)}\right\Vert _{2}^{2} & \leq & C\left(\left\Vert v\right\Vert _{2}^{2}+\left\Vert v\right\Vert _{1}^{2}\right)\nonumber \\
	& \leq & C\left\Vert v\right\Vert _{2}^{2}.
	\end{eqnarray}
	Keeping only the highest Sobolev norms, we have
	\begin{eqnarray}
	\left\Vert \left(uG_{v}\right)^{\left(k\right)}\right\Vert _{2}^{2} & \leq & C_{0}\left\Vert u\right\Vert _{H_{per}^{k}\left(U\right)}^{2}\left\Vert v\right\Vert _{2}^{2}+C_{1}\left\Vert u\right\Vert _{H_{per}^{k}\left(U\right)}^{2}\left\Vert v\right\Vert _{H_{per}^{k-1}\left(U\right)}^{2}\nonumber \\
	& \leq & C\left\Vert u\right\Vert _{H_{per}^{k}\left(U\right)}^{2}\left\Vert v\right\Vert _{H_{per}^{k-1}\left(U\right)}^{2}.
	\end{eqnarray}
	
\qquad\end{proof}

Now, we assume that $\rho_0 \in H^k_{per}$ for some $k\geq 0$ and prove the corresponding regularity of $\rho$. 
\begin{theorem}
	\label{thm:x_regular}(Improved regularity) Let $k\geq0$ and suppose $\rho_{0}\in H_{per}^{k}\left(U\right)$
	with $\rho_{0}\geq0$ and $\int_{U}\rho_{0}\left(x\right)dx=1$. Then
	the unique solution to (\ref{eq:evo_equation}) satisfies 
	\begin{eqnarray}
	\rho & \in & {L}^{2}\left(0,T;H_{per}^{k+1}\left(U\right)\right)\cap{L}^{\infty}\left(0,T;H_{per}^{k}\left(U\right)\right), \nonumber
	\end{eqnarray}
	with the estimate 
	\begin{eqnarray}
	\left\Vert \rho\right\Vert _{{L}^{2}\left(0,T;H_{per}^{k+1}\left(U\right)\right)}+\left\Vert \rho\right\Vert _{{L}^{\infty}\left(0,T;H_{per}^{k}\left(U\right)\right)} & \leq & C\left(\rho_{0};k,T\right), \nonumber
	\end{eqnarray}
	where 
	\[
	C\left(\rho_{0};k,T\right):=C\left(T\right)\left(\sum_{i=0}^{k}\left\Vert \rho_{0}\right\Vert _{H_{per}^{k-i}\left(U\right)}^{2^{i+1}}\right)^{1/2}.
	\]
\end{theorem}
\begin{proof}
	We prove the statements by proving uniform estimates on $\rho_{n}$
	by induction on $k$. The base case $k=0$ is provided in Proposition~\ref{prop:unif_L2_est}.
	The $k=1$ case is Proposition~\ref{prop:unif_L2_est_higher}. Suppose for some $k\geq1$,
	\begin{eqnarray}
	\left\Vert \rho_{n}\right\Vert _{{L}^{2}\left(0,T;H_{per}^{k+1}\left(U\right)\right)}+\left\Vert \rho_{n}\right\Vert _{{L}^{\infty}\left(0,T;H_{per}^{k}\left(U\right)\right)} & \leq & C\left(\rho_{0};k,T\right),\label{eq:inductive_hyp}
	\end{eqnarray}
	for all $n$. We differentiate equation (\ref{eq:sequence_evolution})
	$k$ times with respect to $x$, multiply it by $-\partial^{k+2}_x\rho_{n}$
	and integrate over $U$ to get
	\begin{eqnarray}
	&  & \frac{1}{2}\frac{d}{dt}\left\Vert \partial_{x}^{k+1}\rho_{n}\left(t\right)\right\Vert _{2}^{2}+\frac{\sigma^{2}}{2}\left\Vert \partial_{x}^{k+2}\rho_{n}\left(t\right)\right\Vert _{2}^{2}\nonumber \\
	& \leq & \int_{U}\left|\partial_{x}^{k+1}\rho_{n}\left(t\right)\partial_{x}^{k+1}\left(\rho_{n}G_{\rho_{n-1}}\right)\left(t\right)\right|dx\nonumber \\
	& \leq & \frac{\sigma^{2}}{4}\left\Vert \partial_{x}^{k+2}\rho_{n}\left(t\right)\right\Vert _{2}^{2}+C\left\Vert \partial_{x}^{k+1}\left(\rho_{n}G_{\rho_{n-1}}\right)\left(t\right)\right\Vert _{2}^{2}.
	\end{eqnarray}
	Using Proposition \ref{prop:induction_estimate} with $u=\rho_{n}\left(t\right)$
	and $v=\rho_{n-1}\left(t\right)$, we have 
	\begin{eqnarray}
	&  & \frac{1}{2}\frac{d}{dt}\left\Vert \partial_{x}^{k+1}\rho_{n}\left(t\right)\right\Vert _{2}^{2}+\frac{\sigma^{2}}{4}\left\Vert \partial_{x}^{k+2}\rho_{n}\left(t\right)\right\Vert _{2}^{2}\nonumber \\
	& \leq & C\left(\left\Vert \rho_{n}\left(t\right)\right\Vert _{H_{per}^{k+1}\left(U\right)}^{2}\left\Vert \rho_{n-1}\left(t\right)\right\Vert _{H_{per}^{k}\left(U\right)}^{2}\right)
	\end{eqnarray}
	Integrating over time, we get 
	\begin{eqnarray}
	&  & \underset{0\leq t \leq T}{\mbox{sup}}\left\Vert \rho_{n}\left(t\right)\right\Vert _{H_{per}^{k+1}\left(U\right)}^{2}+\left\Vert \rho_{n}\right\Vert _{{L}^{2}\left(0,T;H_{per}^{k+2}\left(U\right)\right)}^{2}\nonumber \\
	& \leq & \left\Vert \rho_{0}\right\Vert _{H_{per}^{k+1}\left(U\right)}^{2}\nonumber \\
	&  & +C\left(\left\Vert \rho_{n-1}\right\Vert _{{L}^{\infty}\left(0,T;H_{per}^{k}\left(U\right)\right)}^{2}\left\Vert \rho_{n}\right\Vert _{{L}^{2}\left(0,T;H_{per}^{k+1}\left(U\right)\right)}^{2}\right)\nonumber \\
	& \leq & \left\Vert \rho_{0}\right\Vert _{H_{per}^{k+1}\left(U\right)}^{2}+\left[C\left(\rho_{0};k,T\right)\right]^{4}\nonumber \\
	& \leq & C\left(\rho_{0};k+1,T\right)^{2}
	\end{eqnarray}
	This completes the induction. Taking limits, we obtain
	\begin{eqnarray}
	\rho & \in & {L}^{2}\left(0,T;H_{per}^{k+2}\left(U\right)\right)\cap{L}^{\infty}\left(0,T;H_{per}^{k+1}\left(U\right)\right),
	\end{eqnarray}
	with the estimate 
	\begin{eqnarray}
	\left\Vert \rho\right\Vert _{{L}^{2}\left(0,T;H_{per}^{k+2}\left(U\right)\right)}+\left\Vert \rho\right\Vert _{{L}^{\infty}\left(0,T;H_{per}^{k+1}\left(U\right)\right)} & \leq & C\left(\rho_{0};k+1,T\right).
	\end{eqnarray}
	
\qquad\end{proof}

So far we have only considered regularity in space. The same can also
be done in the time domain. 
\begin{theorem}
	\label{thm:reg_xt}(Improved regularity) Let $k\geq0$ and suppose $\rho_{0}\in H_{per}^{2k}\left(U\right)$
	with $\rho_{0}\geq0$ and $\int_{U}\rho_{0}\left(x\right)dx=1$. Then,
	
	(i) For every $0\leq m\leq k$, the unique solution to (\ref{eq:evo_equation})
	satisfies 
	\[
	\frac{d^{m}\rho}{dt^{m}}  \in  {L}^{2}\left(0,T;H_{per}^{2k-2m+1}\left(U\right)\right)\cap{L}^{\infty}\left(0,T;H_{per}^{2k-2m}\left(U\right)\right),
	\]
	
	with the estimate
	\[
	\sum_{m=0}^{k}\left(\left\Vert \frac{d^{m}\rho}{dt^{m}}\right\Vert _{{L}^{2}\left(0,T;H_{per}^{2k-2m+1}\left(U\right)\right)}+\left\Vert \frac{d^{m}\rho}{dt^{m}}\right\Vert _{{L}^{\infty}\left(0,T;H_{per}^{2k-2m}\left(U\right)\right)}\right) \leq D\left(\rho_{0};k,T\right),
	\]
	
	where 
	\[
	D\left(\rho_{0};k,T\right):=\left(\sum_{j=0}^{k}C\left(\rho_{0};2k,T\right)^{2^{j+1}}\right)^{1/2}.
	\]
	
	(ii) Moreover, 
	\[
	\frac{d^{k+1}\rho}{dt^{k+1}}\in{L}^{2}\left(0,T;H_{per}^{-1}\left(U\right)\right),
	\]
	
	with the estimate 
	\[
	\left\Vert \frac{d^{k+1}\rho}{dt^{k+1}}\right\Vert _{{L}^{2}\left(0,T;H_{per}^{-1}\left(U\right)\right)} \leq D\left(\rho_{0};k,T\right).
	\]
\end{theorem}
\begin{proof}
	Let us prove that for all $M\leq k$, 
	\begin{eqnarray}
	&  & \sum_{m=0}^{M}\left(\left\Vert \frac{d^{m}\rho_{n}}{dt^{m}}\right\Vert _{{L}^{2}\left(0,T;H_{per}^{2k-2m+1}\left(U\right)\right)}+\left\Vert \frac{d^{m}\rho_{n}}{dt^{m}}\right\Vert _{{L}^{\infty}\left(0,T;H_{per}^{2k-2m}\left(U\right)\right)}\right)\nonumber \\
	& \leq & C\left(\rho_{0};k,M,T\right),\label{eq:inductive_hyp-1}
	\end{eqnarray}
	where we have defined
	\begin{equation}
	C\left(\rho_{0};k,M,T\right):=\left(\sum_{j=0}^{M}C\left(\rho_{0};2k,T\right)^{2^{j+1}}\right)^{1/2}.
	\end{equation}
	This is done by induction on $M$ up to $k$. The case $M=0$ is Theorem
	\ref{thm:x_regular}. Suppose we have for some $0\leq M<k$ the estimate
	(\ref{eq:inductive_hyp-1}). Differentiating equation (\ref{eq:sequence_evolution})
	$M$ times with respect to $t$ and using the Leibniz rule, we have
	\begin{eqnarray}
	\rho_{n}^{\left(M+1\right)} & = & \frac{\sigma^{2}}{2}\rho_{n_{xx}}^{\left(M\right)}+\left(\rho_{n}G_{\rho_{n-1}}\right)_{x}^{\left(M\right)}\nonumber \\
	& = & \frac{\sigma^{2}}{2}\rho_{n_{xx}}^{\left(M\right)}+C\sum_{m=0}^{M}\left(\rho_{n}^{\left(m\right)}G_{\rho_{n-1}^{\left(M-m\right)}}\right)_{x}.
	\end{eqnarray}
	where we used the shorthand $\rho_{n}^{\left(m\right)}:=\partial^{m}\rho_{n}/\partial t^{m}$.
	Thus, we have
	\begin{eqnarray}
	\left\Vert \rho_{n}^{\left(M+1\right)}\left(t\right)\right\Vert _{H_{per}^{2k-2M-1}\left(U\right)}^{2} & \leq & C_{1}\left\Vert \rho_{n}^{\left(M\right)}\left(t\right)\right\Vert _{H_{per}^{2k-2M+1}\left(U\right)}^{2}\nonumber \\
	&  & +C_{2}\sum_{m=0}^{M}\left\Vert \rho_{n}^{\left(m\right)}\left(t\right)G_{\rho_{n-1}^{\left(M-m\right)}}\left(t\right)\right\Vert _{H_{per}^{2k-2M}\left(U\right)}^{2}.
	\end{eqnarray}
	Using Proposition \ref{prop:induction_estimate} with $u=\rho_{n}^{\left(m\right)}\left(t\right)$
	and $v=\rho_{n-1}^{\left(M-m\right)}\left(t\right)$, we have
	\begin{eqnarray}
	&  & \left\Vert \rho_{n}^{\left(M+1\right)}\left(t\right)\right\Vert _{H_{per}^{2k-2M-1}\left(U\right)}^{2}\nonumber \\
	& \leq & C_{1}\left\Vert \rho_{n}^{\left(M\right)}\left(t\right)\right\Vert _{H_{per}^{2k-2M+1}\left(U\right)}^{2}\nonumber \\
	&  & +C_{2}\sum_{m=0}^{M}\left\Vert \rho_{n}^{\left(m\right)}\left(t\right)G_{\rho_{n-1}^{\left(M-m\right)}}\left(t\right)\right\Vert _{H_{per}^{2k-2M}\left(U\right)}^{2}\nonumber \\
	& \leq & C_{1}\left\Vert \rho_{n}^{\left(M\right)}\left(t\right)\right\Vert _{H_{per}^{2k-2M+1}\left(U\right)}^{2}\nonumber \\
	&  & +C_{2}\sum_{m=0}^{M}\left\Vert \rho_{n}^{\left(m\right)}\left(t\right)\right\Vert _{H_{per}^{2k-2M}\left(U\right)}^{2}\left\Vert \rho_{n-1}^{\left(M-m\right)}\left(t\right)\right\Vert _{H_{per}^{2k-2M-1}\left(U\right)}^{2}\nonumber \\
	& \leq & C_{1}\left\Vert \rho_{n}^{\left(M\right)}\left(t\right)\right\Vert _{H_{per}^{2k-2M+1}\left(U\right)}^{2}\nonumber \\
	&  & +C_{2}\sum_{m=0}^{M}\left\Vert \rho_{n}^{\left(m\right)}\left(t\right)\right\Vert _{H_{per}^{2k-2m+1}\left(U\right)}^{2}\left\Vert \rho_{n-1}^{\left(M-m\right)}\left(t\right)\right\Vert _{H_{per}^{2k-2\left(M-m\right)}\left(U\right)}^{2}.
	\end{eqnarray}
	Integrating over time then gives
	\begin{eqnarray}
	&  & \left\Vert \rho_{n}^{\left(M+1\right)}\right\Vert _{{L}^{2}\left(0,T;H_{per}^{2k-2M-1}\left(U\right)\right)}^{2}\nonumber \\
	& \leq & C_{1}\left\Vert \rho_{n}^{\left(M\right)}\right\Vert _{{L}^{2}\left(0,T;H_{per}^{2k-2M+1}\left(U\right)\right)}^{2}\nonumber \\
	&  & +C_{2}\sum_{m=0}^{M}\left\Vert \rho_{n}^{\left(m\right)}\right\Vert _{{L}^{2}\left(0,T;H_{per}^{2k-2m+1}\left(U\right)\right)}^{2}\left\Vert \rho_{n-1}^{\left(M-m\right)}\right\Vert _{{L}^{\infty}\left(0,T;H_{per}^{2k-2\left(M-m\right)}\left(U\right)\right)}^{2}.
	\end{eqnarray}
	Since $0\leq m,M-m\leq M$, we can apply the inductive hypothesis
	(\ref{eq:inductive_hyp-1}) to conclude that 
	\begin{eqnarray}
	\left\Vert \rho_{n}^{\left(M+1\right)}\right\Vert _{{L}^{2}\left(0,T;H_{per}^{2k-2M-1}\left(U\right)\right)}^{2} & \leq & C\left(\rho_{0};k,M,T\right)^{2}+C\left(\rho_{0};k,M,T\right)^{4}\nonumber \\
	& \leq & C\left(\rho_{0};k,M+1,T\right)^{2}.
	\end{eqnarray}
	Similarly, 
	\begin{eqnarray}
	&  & \left\Vert \rho_{n}^{\left(M+1\right)}\right\Vert _{{L}^{\infty}\left(0,T;H_{per}^{2k-2M-2}\left(U\right)\right)}^{2}\nonumber \\
	& \leq & C_{1}\left\Vert \rho_{n}^{\left(M\right)}\right\Vert _{{L}^{\infty}\left(0,T;H_{per}^{2k-2M}\left(U\right)\right)}^{2}\nonumber \\
	&  & +C_{2}\sum_{m=0}^{M}\left\Vert \rho_{n}^{\left(m\right)}\right\Vert _{{L}^{\infty}\left(0,T;H_{per}^{2k-2m}\left(U\right)\right)}^{2}\left\Vert \rho_{n-1}^{\left(M-m\right)}\right\Vert _{{L}^{\infty}\left(0,T;H_{per}^{2k-2\left(M-m\right)}\left(U\right)\right)}^{2}.\nonumber \\
	& \leq & C\left(\rho_{0};k,M+1,T\right)^{2}.
	\end{eqnarray}
	This completes the induction on $M$ up to $k$. Putting $M=k$ into
	(\ref{eq:inductive_hyp-1}) and taking limits proves part (i).
	
	To prove the second part, notice that 
	\begin{equation}
	\rho_{n}^{\left(k+1\right)}=\left(\frac{\sigma^{2}}{2}\rho_{n_{x}}^{\left(k\right)}+C\sum_{m=0}^{k}\left(\rho_{n}^{\left(m\right)}G_{\rho_{n-1}^{\left(k-m\right)}}\right)\right)_{x}.
	\end{equation}
	Hence,
	\begin{eqnarray}
	&  & \left\Vert \rho_{n}^{\left(k+1\right)}\right\Vert _{{L}^{2}\left(0,T;H_{per}^{-1}\left(U\right)\right)}^{2}\nonumber \\
	& \leq & \left\Vert \frac{\sigma^{2}}{2}\rho_{n_{x}}^{\left(k\right)}+C\sum_{m=0}^{k}\left(\rho_{n}^{\left(m\right)}G_{\rho_{n-1}^{\left(k-m\right)}}\right)\right\Vert _{{L}^{2}\left(0,T;{L}_{per}^{2}\left(U\right)\right)}^{2}\nonumber \\
	& \leq & C_{1}\left\Vert \rho_{n}^{\left(k\right)}\right\Vert _{{L}^{2}\left(0,T;H_{per}^{1}\left(U\right)\right)}^{2}\nonumber \\
	&  & +C_{2}\sum_{m=0}^{k}\left\Vert \rho_{n}^{\left(m\right)}\right\Vert _{{L}^{\infty}\left(0,T;{L}_{per}^{2}\left(U\right)\right)}^{2}\left\Vert \rho_{n-1}^{\left(M-m\right)}\right\Vert _{{L}^{2}\left(0,T;{L}_{per}^{2}\left(U\right)\right)}^{2}\nonumber \\
	& \leq & D\left(\rho_{0};k,T\right).
	\end{eqnarray}
	Taking limits then proves part (ii). \qquad\end{proof}

\begin{corollary}\label{cor:last_reg}
	Let $T>0$ and $\rho_{0}\in H_{per}^{3}\left(U\right)$ with $\rho_{0}\geq0$
	and $\int_{U}\rho_{0}\left(x\right)dx=1$. Then the unique solution
	to (\ref{eq:evo_equation}) satisfies
	\[
	\rho\in C^{1}\left(0,T;C_{per}^{2}\left(U\right)\right),
	\]
	after possibly being redefined on a set of measure zero. \end{corollary}
\begin{proof}
	By Theorem \ref{thm:x_regular}, $\rho\in{L}^{\infty}\left(0,T;H_{per}^{3}\left(U\right)\right)$,
	i.e. $\rho_{xx}\in{L}^{\infty}\left(0,T;H_{per}^{1}\left(U\right)\right)$.
	Hence there exists a version of $\rho$ with $\rho_{xx}\left(t\right)\in C_{per}^{0,\frac{1}{2}}\left(U\right)$,
	so that in particular, $\rho\left(t\right)\in C_{per}^{2}\left(U\right)$.
	Next, using Theorem \ref{thm:reg_xt}, we have $\rho_{t}\in{L}^{2}\left(0,T;H_{per}^{1}\left(U\right)\right)$
	and $\rho_{tt}\in{L}^{2}\left(0,T;H_{per}^{-1}\left(U\right)\right)$,
	hence by Theorem \ref{thm:Evans} there is a version of $\rho$ so
	that $\rho_{t}\in C\left(0,T;{L}_{per}^{2}\left(U\right)\right)$.
	Hence we have 
	\begin{equation}
	\rho\in C^{1}\left(0,T;C_{per}^{2}\left(U\right)\right),
	\end{equation}
	up to a set of measure zero. 
\qquad\end{proof}

This result allows us to restate the results in Section \ref{sec:apriori}
without the a priori smoothness assumption. We summarize the main results of this paper in the following:
\begin{theorem} \label{thm:mainresult}
	Let $\rho_{0}\in H_{per}^{3}\left(U\right)$ with $\rho_{0}\geq0$
	and $\int_{U}\rho_{0}\left(x\right)dx=1$. Then, there exists a unique
	weak solution $\rho$ to equation (\ref{eq:evo_equation}), with
	
	(i) (Regularity) $\rho\in C^{1}\left(0,\infty;C_{per}^{2}\left(U\right)\right),$
	
	(ii) (Nonnegativity) $\rho\left(t\right)\geq0$ for all $t\geq0$.
	
	(iii) (Stability) Furthermore, if $\sigma^2>\frac{2\ell}{\pi} \left(2R+{R^2}/{\sqrt{3}\, \ell}\right)$,
	then $\rho\left(t\right)\rightarrow\frac{1}{2\ell}$ in ${L}_{per}^{2}$
	exponentially as $t\rightarrow\infty$. \end{theorem}
\begin{proof}
	Existence and uniqueness follow from Theorem~\ref{thm:Existnece-1}. (i) follows from Corollary~\ref{cor:last_reg}. Having established (i), (ii) and (iii) then follows from Corollary~\ref{cor:positivity} and Theorem~\ref{thm:apriori_stability} respectively. 
\qquad\end{proof}

\section*{Acknowledgement} 
This work was completed while Q. Jiu was visiting the Department of Mathematics at Princeton University. The authors are grateful for many discussions with Prof. Weinan E. 

\footnotesize{
	\bibliography{hk}

\begin{thebibliography}{10}

\bibitem{axelrod1997complexity}
{\sc Robert Axelrod}, {\em The evolution of cooperation}, Basic Books, 2006.

\bibitem{bhattacharyya2013convergence}
{\sc Arnab Bhattacharyya, Mark Braverman, Bernard Chazelle, and Huy~L Nguyen},
  {\em On the convergence of the hegselmann-krause system}, in Proceedings of
  the 4th conference on Innovations in Theoretical Computer Science, ACM, 2013,
  pp.~61--66.

\bibitem{blondel20072r}
{\sc Vincent~D Blondel, Julien~M Hendrickx, and John~N Tsitsiklis}, {\em On the
  2r conjecture for multi-agent systems}, in Control Conference (ECC), 2007
  European, IEEE, 2007, pp.~874--881.

\bibitem{blondel2009krause}
{\sc Vincent~D. Blondel, Julien~M. Hendrickx, and John~N. Tsitsiklis}, {\em On
  {K}rause's multi-agent consensus model with state-dependent connectivity},
  Automatic Control, IEEE Transactions on, 54 (2009), pp.~2586--2597.

\bibitem{canutoFT2012}
{\sc Claudio Canuto, Fabio Fagnani, and Paolo Tilli}, {\em An {E}ulerian
  approach to the analysis of {K}rause's consensus models}, SIAM J. Control
  Optim., 50 (2012), pp.~243--265.

\bibitem{castellano2009statistical}
{\sc Claudio Castellano, Santo Fortunato, and Vittorio Loreto}, {\em
  Statistical physics of social dynamics}, Reviews of modern physics, 81
  (2009), p.~591.

\bibitem{chazelle2012dynamics}
{\sc Bernard Chazelle}, {\em The dynamics of influence systems}, in Foundations
  of Computer Science (FOCS), 2012 IEEE 53rd Annual Symposium on, IEEE, 2012,
  pp.~311--320.

\bibitem{chazelle2015Algo}
\leavevmode\vrule height 2pt depth -1.6pt width 23pt, {\em An algorithmic
  approach to collective behavior}, Journal of Statistical Physics, 158 (2015),
  pp.~514--548.

\bibitem{DBLP:journals/corr/ChazelleW15}
{\sc Bernard Chazelle and Chu Wang}, {\em Inertial {H}egselmann-{K}rause
  systems}, CoRR, abs/1502.03332 (2015).

\bibitem{easley2012networks}
{\sc David Easley and Jon Kleinberg}, {\em Networks, Crowds, and Markets:
  Reasoning About a Highly Connected World}, Cambridge Univ. Press, 2010.

\bibitem{Evans}
{\sc L.C. Evans}, {\em Partial Differential Equations}, American Mathematical
  Society, 1998.

\bibitem{fortunato2005consensus}
{\sc Santo Fortunato}, {\em On the consensus threshold for the opinion dynamics
  of {K}rause-{H}egselmann}, International Journal of Modern Physics C, 16
  (2005), pp.~259--270.

\bibitem{DBLP:journals/corr/GarnierPY15}
{\sc Josselin Garnier, George Papanicolaou, and Tzu{-}Wei Yang}, {\em Consensus
  convergence with stochastic effects}, CoRR, abs/1508.07313 (2015).

\bibitem{hegselmann2002opinion}
{\sc Rainer Hegselmann and Ulrich Krause}, {\em Opinion dynamics and bounded
  confidence models, analysis, and simulation}, Journal of Artificial Societies
  and Social Simulation, 5 (2002).

\bibitem{hegselmann2015opinion}
\leavevmode\vrule height 2pt depth -1.6pt width 23pt, {\em Opinion dynamics
  under the influence of radical groups, charismatic leaders, and other
  constant signals: A simple unifying model}, Networks and Heterogeneous Media,
  10 (2015), pp.~477--509.

\bibitem{hendrickx2006convergence}
{\sc Julien Hendrickx and Vincent Blondel}, {\em Convergence of different
  linear and non-linear {V}icsek models}, in Proc. 17th Int. Symp. Math. Theory
  Networks Syst. (MTNS 2006), 2006, pp.~1229--1240.

\bibitem{jadbabaieLM03}
{\sc Ali Jadbabaie, Jie Lin, and A.~Stephen Morse}, {\em Coordination of groups
  of mobile autonomous agents using nearest neighbor rules}, IEEE Trans.
  Automatic Control, 48 (2003), pp.~988--1001.

\bibitem{krause2000discrete}
{\sc Ulrich Krause}, {\em A discrete nonlinear and non-autonomous model of
  consensus formation}, in Communications in difference equations, Gordon and
  Breach, Amsterdam, 2000, pp.~227--236.

\bibitem{lorenz2005stabilization}
{\sc Jan Lorenz}, {\em A stabilization theorem for dynamics of continuous
  opinions}, Physica A: Statistical Mechanics and its Applications, 355 (2005),
  pp.~217--223.

\bibitem{martinez2007synchronous}
{\sc Sonia Mart{\'\i}nez, Francesco Bullo, Jorge Cort{\'e}s, and Emilio
  Frazzoli}, {\em On synchronous robotic networks -- part ii: Time complexity
  of rendezvous and deployment algorithms}, Automatic Control, IEEE
  Transactions on, 52 (2007), pp.~2214--2226.

\bibitem{moreau2005stability}
{\sc Luc Moreau}, {\em Stability of multiagent systems with time-dependent
  communication links}, Automatic Control, IEEE Transactions on, 50 (2005),
  pp.~169--182.

\bibitem{motsch2014heterophilious}
{\sc Sebastien Motsch and Eitan Tadmor}, {\em Heterophilious dynamics enhances
  consensus}, SIAM review, 56 (2014), pp.~577--621.

\bibitem{pineda2013noisy}
{\sc Miguel Pineda, Ra{\'u}l Toral, and Emilio Hern{\'a}ndez-Garc{\'\i}a}, {\em
  The noisy {H}egselmann-{K}rause model for opinion dynamics}, The European
  Physical Journal B, 86 (2013), pp.~1--10.

\bibitem{ramirez2006follow}
{\sc Daniel Ramirez-Cano and Jeremy Pitt}, {\em Follow the leader: Profiling
  agents in an opinion formation model of dynamic confidence and individual
  mind-sets}, in Intelligent Agent Technology, 2006. IAT'06. IEEE/WIC/ACM
  International Conference on, IEEE, 2006, pp.~660--667.

\bibitem{touri2011discrete}
{\sc Behrouz Touri and Angelia Nedi\'{c}}, {\em Discrete-time opinion
  dynamics}, in 2011 Conference Record of the Forty-Fifth Asilomar Conference
  on Signals, Systems and Computers (ASILOMAR), 2011.

\bibitem{wang2015inprep}
{\sc Chu Wang, Qianxiao Li, Weinan E, and Bernard Chazelle}, {\em In
  preparation},  (2015).

\bibitem{wongkaew2015control}
{\sc Suttida Wongkaew, Marco Caponigro, and Alfio Borz{\`\i}}, {\em On the
  control through leadership of the {H}egselmann-{K}rause opinion formation
  model}, Mathematical Models and Methods in Applied Sciences, 25 (2015),
  pp.~565--585.

\end{thebibliography}
	\bibliographystyle{siam}}
\end{document}